\documentclass[12pt]{amsproc}
\usepackage{amscd,amssymb,graphicx,xy}
\xyoption{all}

\newtheorem{theorem}{Theorem}[section]

\newtheorem{lemma}[theorem]{Lemma}
\newtheorem{proposition}[theorem]{Proposition}

\theoremstyle{definition}
\newtheorem{definition}[theorem]{Definition}
\newtheorem{notation}[theorem]{Notation}
\newtheorem{example}[theorem]{Example}

\theoremstyle{remark}
\newtheorem{remark}[theorem]{Remark}

\numberwithin{equation}{section}

\DeclareMathOperator{\comp}{\#}
\renewcommand{\d}{\partial}
\DeclareMathOperator{\e}{\epsilon}
\DeclareMathOperator{\Hom}{Hom}
\DeclareMathOperator{\id}{id}
\DeclareMathOperator{\im}{im}
\DeclareMathOperator{\Or}{\mathcal{O}}
\renewcommand{\v}{\vee}
\newcommand{\w}{\wedge}
\newcommand{\Z}{\mathbf{Z}}

\begin{document}

\title{The algebra of the nerves of omega-categories}

\author{Richard Steiner}

\address{School of Mathematics and Statistics, University of
Glasgow, University Gardens, Glasgow, Great Britain G12 8QW}

\email{Richard.Steiner@glasgow.ac.uk}



\begin{abstract}
We show that the nerve of a strict omega-category can be described algebraically as a simplicial set with additional operations subject to certain identities. The resulting structures are called sets with complicial identities. We also construct an equivalence between the categories of strict omega-categories and of sets with complical identities.
\end{abstract}

\maketitle

\section{Introduction} \label{S1}

This paper is concerned with the simplicial nerves of strict $\omega$-categories, as constructed by Street~\cite{B6}. The nerves are simplicial sets with additional structure, and the problem is to characterise the additional structure which can occur. One characterisation, due to Verity~\cite{B7}, says that the nerves are complicial sets; that is to say, they have distinguished classes of thin elements satisfying certain axioms. The object of this paper is to give a more concrete algebraic characterisation: the nerves are simplicial sets with additional operations satisfying certain identities. The result is a set with complicial identities as defined in~\cite{B5}. The resultant characterisation is like the characterisation of cubical nerves given by Al-Agl, Brown and Steiner~\cite{B1}.

The method involves a comparison of the theories of $\omega$-categories and of sets with complicial identities in the technical sense of universal algebra. It turns out that both theories can be expressed in terms of chain complexes and chain maps. The theory of $\omega$-categories is represented by simple chain complexes~\cite{B4}; the theory of sets with complicial identities is represented by the chain complexes of simplexes and by certain colimits of these chain complexes. The proof is based on relationships between the various chain complexes involved.

The paper is structured as follows. In Section~\ref{S2} we describe $\omega$-categories and show that their theory is represented by the class of simple $\omega$-categories (see~\cite{B2}). In Section~\ref{S3} we describe sets with complicial identities. In Section~\ref{S4} we describe a category of chain complexes with additional structure called augmented directed complexes and a functor~$\nu$ from this category to the category of $\omega$-categories. In Section~\ref{S5} we show that simple $\omega$-categories are the images under~$\nu$ of simple chain complexes; it follows that the theory of $\omega$-categories can be described in terms of simple chain complexes. In Section~\ref{S6} we construct a functor~$\lambda$ from augmented directed complexes to sets with complicial identities, and in Section~\ref{S7} we use this functor to express the theory of sets with complicial identities in terms of the chain complexes of simplexes. We have now described the categories of $\omega$-categories and of sets with complicial identities in terms of augmented directed complexes, and can therefore compare the two categories. The comparison occupies Sections \ref{S8}--\ref{S13}.

The idea behind the comparison is as follows. Let $X$ be a set with complicial identities; then there is a contravariant functor from simple chain complexes to sets given by
$$S\mapsto\Hom[\lambda S,X].$$
This functor will yield an $\omega$-category provided that it takes certain colimit diagrams to limit diagrams; we therefore need information about the sets $\Hom[\lambda S,X]$. We obtain this information by showing that $S$~is a retract of the chain complex of a simplex. We begin in Section~\ref{S8} by showing that $S$~is a quotient of the chain complex of a simplex. We then show that $S$~is a retract by constructing an idempotent endomorphism of the chain complex of the simplex with the appropriate kernel. This endomorphism represents an operation in sets with complicial identities. We construct the operation in Section~\ref{S9} and give some computations concerning the induced endomorphism in Section~\ref{S10}; we prove that the corresponding endomorphism is idempotent with the correct kernel in Section~\ref{S11}; we show that the required diagrams are limit diagrams in Section~\ref{S12}. In Section~\ref{S13} we deduce the main result (Theorem~\ref{T13.3}): $\omega$-categories are equivalent to sets with complicial identities.
 
\section{The theory of $\omega$-categories} \label{S2}

In this paper all $\omega$-categories are strict $\omega$-categories. We will use an algebraic definition with infinitely many sorts, as follows. 

\begin{definition} \label{D2.1}
An \emph{$\omega$-category}~$C$ is a sequence of sets $C_0,C_1,\ldots\,$ together with the following structure. 

(1) If $x\in C_p$ then there are \emph{identity elements}
$$i_p^n x\in C_n\quad (p<n),$$
\emph{sources}
$$d_q^- x\in C_q\quad (q<p),$$
and \emph{targets}
$$d_q^+ x\in C_q\quad (q<p).$$

(2) If $x,y\in C_p$ and $d_q^+ x=d_q^- y$ for some $q<p$ then there is a \emph{composite}
$$x\comp_q y\in C_p.$$

(3) If $x\in C_p$ then
\begin{align*}
&i_m^n i_p^m x=i_p^n x\quad (p<m<n),\\
&d_p^- i_p^n x=d_p^+ i_p^n x=x\quad (p<n),\\
&i_m^p d_m^- x\comp_m x=x\comp_m i_m^p d_m^+ x=x\quad (m<p),\\
&d_m^- d_n^- x=d_m^- d_n^+ x=d_m^- x\quad (m<n<p),\\
&d_m^+d_n^- x=d_m^+ d_n^+ x=d_m^+ x\quad (m<n<p).
\end{align*}

(4) If $x,y\in C_p$ and $d_q^+x=d_q^- y$ for some $q<p$ then
\begin{align*}
&i_p^n(x\comp_q y)=i_p^n x\comp_q i_p^n y\quad (p<n),\\
&d_m^-(x\comp_q y)=d_m^- x\comp_q d_m^- y\quad (q<m<p),\\
&d_m^+(x\comp_q y)=d_m^+ x\comp_q d_m^+ y\quad (q<m<p),\\
&d_q^-(x\comp_q y)=d_q^- x,\\
&d_q^+(x\comp_q y)=d_q^+ y.
\end{align*}

(5) If $x,y,z\in C_p$ and $d_q^+ x=d_q^- y$, $d_q^+ y=d_q^- z$ for some $q<p$ then 
$$(x\comp_q y)\comp_q z= x\comp_q(y\comp_q z).$$

(6) If $x,y,z,w\in C_p$ and $d_q^+ x=d_q^- y$, $d_m^+ y=d_m^- z$, $d_q^+z=d_q^- w$ with $m<q<p$ then
$$(x\comp_q y)\comp_m(z\comp_q w)=(x\comp_m z)\comp_q(y\comp_m w).$$

A \emph{morphism of $\omega$-categories} $f\colon C\to D$ is a sequence of functions $f\colon C_p\to D_p$ commuting with the identity, source, target and composition operations.
\end{definition}

\begin{remark} \label{R2.2}
In an $\omega$-category the identity element functions~$i_p^n$ must be injective. It is therefore possible to require them to be inclusions, yielding a one-sorted description with operations $d_q^-,d_q^+,\comp_q$.
\end{remark}

\begin{remark} \label{R2.3}
In axiom~(6) the hypotheses $d_q^+ x=d_q^- y$, $d_q^+ z=d_q^- w$ imply that
$$d_m^+ x=d_m^+ d_q^+ x=d_m^+ d_q^- y=d_m^+ y,\quad
d_m^- z=d_m^- d_q^+ z=d_m^- d_q^- w=d_m^- w.$$
One could therefore replace the single equality $d_m^+ y=d_m^- z$ by the two equalities $d_m^+ x=d_m^- z$, $d_m^+ y=d_m^- w$. This produces the more usual form of the hypotheses.
\end{remark}

The domains for the axioms can be naturally indexed by sequences of nonnegative integers as follows: for (1)~and~(3) use the one-term sequence~$(p)$; for (2)~and~(4) use $(p,q,p)$ with $p>q$; for~(5) use $(p,q,p,q,p)$ with $p>q$; for~(6) use $(p,q,p,m,p,q,p)$ with $p>q>m$.
All of these sequences are up-down vectors in the sense of the following definition (taken from \cite{B2}, 2.3).

\begin{definition} \label{D2.4}
An \emph{up-down vector} is a non-empty finite sequence of nonnegative integers
$$(p_0,q_1,p_1,\ldots,p_{k-1},q_k,p_k)$$
such that $p_{i-1}>q_i$ and $q_i<p_i$ for $1\leq i\leq k$.
\end{definition}

The corresponding $\omega$-categories are also taken from~\cite{B2} and may be defined as follows.

\begin{definition} \label{D2.5}
Let $\mathbf{s}$ be an up-down vector given by
$$\mathbf{s}=(p_0,q_1,p_1,\ldots,p_{k-1},q_k,p_k).$$
Then an \emph{$\mathbf{s}$-simple $\omega$-category} is an $\omega$-category with a presentation of the following form: the generators form an ordered list $g_0,\ldots,g_k$ with $\dim g_i=p_i$; the relations are given by
$$d_{q_i}^+ g_{i-1}=d_{q_i}^- g_i\quad (1\leq i\leq k).$$
\end{definition}

We will usually treat simple $\omega$-categories as iterated push-outs, using induction on the numbers of terms in up-down vectors. An up-down vector~$\mathbf{s}$ with more than one term will therefore be written in the form
$$\mathbf{s}=(\mathbf{s'},q,p),$$
so that $\mathbf{s'}$~is a shorter up-down vector with last term greater than~$q$ and so that $p$~is an integer greater than~$q$. The corresponding push-outs are as follows.

\begin{definition} \label{D2.6}
Let $\mathbf{s}$ be an up-down vector with more than one term given by $\mathbf{s}=(\mathbf{s'},q,p)$; then an \emph{$\mathbf{s}$-simple square of $\omega$-categories} is a push-out square
$$\xymatrix{
C^0 \ar_{\pi}[d] \ar^{\rho}[r] & C'' \ar^{\tau}[d] \\
C' \ar_{\sigma}[r] & C
}$$
such that $C'$, $C^0$ and~$C''$ are $\mathbf{s'}$-simple, $(q)$-simple and $(p)$-simple with final generators $g'$, $g^0$ and~$g''$ and such that
$$\pi g^0=d_q^+ g',\quad \rho g^0=d_q^- g''.$$
\end{definition}

Obviously we have the following result.

\begin{proposition} \label{P2.7}
If $\tau\colon C''\to C$ is the right hand vertical morphism in an $\mathbf{s}$-simple square of $\omega$-categories and if $C''$~has final generator~$g''$ then $C$~is an $\mathbf{s}$-simple $\omega$-category with final generator $\tau g''$.
\end{proposition}

In the axioms for $\omega$-categories the domains are free $\omega$-categories on single generators and pull-backs corresponding to simple squares. We therefore get the following result.

\begin{proposition} \label{P2.8}
Let $\omega$-$\mathbf{cat}$ be the category of $\omega$-categories, let $\Theta$ be the full-subcategory of simple $\omega$-categories, and let $\hat\Theta$ be the category of contravariant functors from~$\Theta$ to sets which take simple squares to pull-back squares. Then there is an equivalence of categories
$$C\mapsto\Hom(-,C)\colon\textup{$\omega$-$\mathbf{cat}$}\to\hat\Theta.$$
\end{proposition}

\section{Sets with complicial identities} \label{S3}

In this section we recall the definition of sets with complicial identities from~\cite{B5}. A set with complicial identities is a simplicial set $X_0,X_1,\ldots\,$ together with additional partial binary wedge operations~$\w_i$. These operations raise dimension by~$1$; they correspond to the projection of an $(m+1)$-simplex onto the union of the $m$-faces opposite vertices $i$~and $i+2$. The identities are stated here without comment, but there are illustrations in Section~\ref{S7}.

\begin{definition} \label{D3.1}
A \emph{set with complicial identities}~$X$ is a sequence of sets $$X_0,X_1,\ldots$$ 
together with the following structure. 

(1) If $x\in X_m$ then there are \emph{faces}
$$\d_i x\in X_{m-1}\quad (m>0,\ 0\leq i\leq m)$$
and \emph{degeneracies}
$$\e_i x\in X_{m+1}\quad (0\leq i\leq m).$$

(2) If $x,y\in X_m$ and if $\d_i x=\d_{i+1}y$ for some~$i$ with $0\leq i\leq m-1$ then there is a \emph{wedge}
$$x\w_i y\in X_{m+1}.$$

(3) If $x\in X_m$ then
\begin{align*}
&\d_i\d_j x=\d_{j-1}\d_i x\quad (m\geq 2,\ 0\leq i<j\leq m),\\
&\d_i\e_j x=\e_{j-1}\d_i x\quad (0\leq i<j\leq m),\\
&\d_j\e_j x=\d_{j+1}\e_j x=x,\\
&\d_i\e_j x=\e_j\d_{i-1}x\quad (j+2\leq i\leq m+1),\\
&\e_i\e_j x=\e_{j+1}\e_i x\quad (0\leq i\leq j\leq m),\\
&\e_i x=\e_i\d_{i+1}x\w_i x\quad (0\leq i<m),\\
&\e_{i+1}x=x\w_i\e_i\d_i x\quad (0\leq i<m).
\end{align*}

(4) If $x,y\in X_m$ and if $\d_i x=\d_{i+1}y$ with $0\leq i<m$ then
\begin{align*}
&\d_j(x\w_i y)=\d_j x\w_{i-1}\d_j y\quad (0\leq j\leq i-1),\\
&\d_i(x\w_i y)=y,\\
&\d_{i+2}(x\w_i y)=x,\\
&\d_j(x\w_i y)=\d_j x\w_{i-1}\d_j y\quad (i+3\leq j\leq m+1).
\end{align*}

(5) If $b\in X_{m+1}$ and $y,z\in X_m$, if $\d_i y=\d_{i+1}z$ and $\d_i b=\d_{i+1}(y\w_i z)$ with $0\leq i<m$, and if $A=b\w_i(y\w_i z)$, then
$$A=(\d_{i+2}b\w_i y)\w_{i+1}\d_{i+1}A.$$

(6) If $x,y\in X_m$ and $c\in X_{m+1}$, if $\d_i x=\d_{i+1}y$ and $\d_{i+1}(x\w_i y)=\d_{i+2}c$ with $0\leq i<m$, and if $A=(x\w_i y)\w_{i+1}c$, then
$$A=\d_{i+2}A\w_i(y\w_i\d_i c).$$

(7) If $x,y,z\in X_m$ and if $\d_i x=\d_{i+1}y$, $\d_i y=\d_{i+1}z$ with $0\leq i<m$ then
$$[x\w_i\d_{i+1}(y\w_i z)]\w_i(y\w_i z)=(x\w_i y)\w_{i+1}[\d_{i+1}(x\w_i y)\w_i z].$$

(8) If $x,y,z,w\in X_m$, if $\d_{i+1}x=\d_{i+2}y$, $\d_i y=\d_{i+1}z$, $\d_{i+1}w=\d_{i+1}(\d_i x\w_i\d_{i+2}z)$ with $0\leq i\leq m-2$, and if $A=\d_{i+2}[(x\w_{i+1}y)\w_{i+1}(y\w_i z)]$, then
$$A\w_i(w\w_{i+1}\d_i A)=(\d_{i+3}A\w_i w)\w_{i+2}A.$$

(9) If $x,y,z,w\in X_m$ and $\d_i x=\d_{i+1}y$, $\d_i z=\d_{i+1}w$, $\d_{j-1}x=\d_j z$, $\d_{j-1}y=\d_j w$ with $0\leq i\leq j-3\leq m-3$ then
$$(x\w_i y)\w_j(z\w_i w)=(x\w_{j-1}z)\w_i(y\w_{j-1}w).$$

A \emph{morphism of sets with complicial identities} $f\colon X\to Y$ is a sequence of functions $f\colon X_m\to Y_m$ commuting with the face, degeneracy and wedge operations.
\end{definition}

\section{Augmented directed complexes} \label{S4}

In this section we give some definitions and results based on~\cite{B3}.

\begin{definition} \label{D4.1}
An \emph{augmented directed complex} is an augmented chain complex of abelian groups 
$$\xymatrix{\ldots \ar^{\d}[r] & K_1 \ar^{\d}[r] & K_0 \ar^{\e}[r] & \Z,}$$
together with a prescribed submonoid for each chain group~$K_q$. A \emph{morphism of augmented directed complexes} is an augmentation-preserving chain map which takes prescribed submonoids into prescribed submonoids. A \emph{free augmented directed complex} is an augmented directed complex such that each chain group is a free abelian group with a prescribed basis and such that each prescribed submonoid is generated as a monoid by the prescribed basis elements.
\end{definition} 

Let $K$ be a free augmented directed complex.  We note that the prescribed basis elements are uniquely determined as the indecomposable elements in the prescribed submonoids. We regard the union of the prescribed bases for the individual chain groups~$K_q$ as a prescribed graded basis for the entire chain complex~$K$. Given a chain~$c$ in~$K$, we write $\d^+ c$ and $\d^- c$ for the positive and negative parts of the boundary $\d c$; in other words, $\d^+ c$ and $\d^- c$ are the sums of basis elements without common terms such that
$$\d c=\d^+ c-\d^- c.$$

\begin{definition} \label{D4.2}
A \emph{totally ordered directed complex} is a free augmented chain complex together with a total ordering of the basis such that each basis element~$a$ satisfies the following conditions.

(1) In the ordered basis, $a$~appears after the terms of $\d^- a$ and before the terms of $\d^+ a$.

(2) If the dimension of~$a$ is~$p$, then
$$\e(\d^-)^p a=\e(\d^+)^p a=1.$$
\end{definition}

Given an augmented directed complex~$K$, we define an $\omega$-category $\nu K$ as follows. The set $(\nu K)_p$ of $p$-dimensional elements consists of the double sequences
$$(\,x_0^-,x_0^+\mid x_1^-,x_1^+\mid\ldots\,)$$
such that $x_i^-$~and~$x_i^+$ are $i$-dimensional members of the prescribed submonoids, such that 
$$x_i^-=x_i^+=0$$
for $i>p$, such that
$$\e x_0^-=\e x_0^+=1,$$
and such that
$$x_i^+ - x_i^-=\d x_{i+1}^-=\d x_{i+1}^+$$
for $i\geq 0$. For $n>p$ the identity element function $i_p^n\colon (\nu K)_p\to(\nu K)_n$ is the inclusion. For $q<p$, if $x=(\,x_0^-,x_0^+\mid\ldots\,)$ as above, then
$$d_q^\alpha x=(\,x_0^-,x_0^+\mid\ldots\mid x_{q-1}^-,x_{q-1}^+\mid
 x_q^\alpha,x_q^\alpha\mid 0,0\mid\ldots\,).$$
If $x$~and~$y$ are $p$-dimensional and if $d_q^+x=d_q^-y=z$ with $q<p$ then
$$x\comp_q y=x-i_q^p z+y.$$

In particular let $K$ be a totally ordered directed complex and let $a$ be a $p$-dimensional basis element for~$K$; then there is a $p$-dimensional element~$\langle a\rangle$ of $\nu K$, called an \emph{atom}, which is given by
$$\langle a\rangle
 =\bigl(\,(\d^-)^p a,(\d^+)^p a\mid\ldots\mid \d^- a,\d^+ a\mid a,a\mid
 0,0\mid\ldots\,).$$

The main results (\cite{B3}, Theorems 5.11 and~6.1) can be stated as follows.
 
\begin{theorem} \label{T4.3}
The functor~$\nu$ is a fully faithful functor from the category of totally ordered directed complexes to the category of $\omega$-categories.
\end{theorem}

\begin{theorem} \label{T4.4}
Let $K$ be a totally ordered directed complex. Then the $\omega$-category $\nu K$ has a presentation as follows. The generators are the atoms, such that $\langle a\rangle$~is a $p$-dimensional member of $\nu K$ if $a$~is a $p$-dimensional basis element. For each basis element~$a$ of positive dimension~$p$ there are relations
$$d_{p-1}^-\langle a\rangle =w^-(a),\quad
d_{p-1}^+\langle a\rangle =w^+(a),$$
where $w^-(a)$ and $w^+(a)$ are arbitrarily chosen expressions for $d_{p-1}^-\langle a\rangle$ and $d_{p-1}^+\langle a\rangle$ as iterated composites of atoms of dimension less than~$p$.
\end{theorem}

\section{Simple chain complexes} \label{S5}

We will now describe a class of chain complexes corresponding to simple $\omega$-categories. The class was defined in~\cite{B4}. For present purposes it is convenient to proceed inductively.

\begin{definition} \label{D5.1}
Let $\mathbf{s}$ be a one-term up-down vector given by $\mathbf{s}=(p)$. Then an \emph{$\mathbf{s}$-simple chain complex with (final) generator~$a$} is a free augmented directed complex with a $p$-dimensional basis element~$a$ such that the basis elements can be listed as
$$(\d^-)^p a,\ (\d^-)^{p-1}a,\ \ldots,\ \d^- a,\ a,\ 
\d^+ a,\ \ldots,\ (\d^+)^{p-1}a,\ (\d^+)^p a$$
and such that $\e(\d^-)^p a=\e(\d^+)^p a=1$.
\end{definition}

\begin{definition} \label{D5.2}
Let $\mathbf{s}$ be an up-down vector with more than one term given by $\mathbf{s}=(\mathbf{s'},q,p)$, and let $p'$ be the last term in~$\mathbf{s'}$. Then an \emph{$\mathbf{s}$-simple chain complex with final generator~$a$} is an augmented directed complex~$K$ if there are $\mathbf{s'}$-simple, $(q)$-simple and $(p)$-simple subcomplexes $K'$, $K^0$ and~$K''$ with final generators $a'$, $a^0$ and~$a$ such that
\begin{align*}
&K=K'+K'',\\
&K'\cap K''=K^0,\\
&(\d^+)^{p'-q}a'=a^0=(\d^-)^{p-q}a,
\end{align*}
and the distinguished submonoid of~$K$ is the sum of the distinguished submonoids of $K'$~and~$K''$.
\end{definition}

\begin{proposition} \label{P5.3}
Let $\mathbf{s}$ be an up-down vector with last term~$p$ and let $K$ be an $\mathbf{s}$-simple chain complex with final generator~$a$. Then $K$~is a totally ordered directed complex whose ordered basis finishes with the elements
$$a,\ \d^+a,\ \ldots,\ (\d^+)^p a.$$
\end{proposition}

\begin{proof}
The proof is by induction on the number of terms in~$\mathbf{s}$. In the case $\mathbf{s}=(p)$ the result is obvious. From now on, let $\mathbf{s}=(\mathbf{s'},q,p)$,
let $K'$, $K^0$ and~$K''$ be the subcomplexes as in the definition, and let $a^0=(\d^-)^{p-q}a$. It follows from the inductive hypothesis that $K'$~is a totally ordered directed complex whose ordered basis finishes with the terms
$$a^0,\ \d^+ a^0,\ \ldots,\ (\d^+)^q a^0.$$
We observe that $K''$~is obtained from~$K^0$ by adjoining the elements
$$(\d^-)^{p-q-1}a,\ \ldots,\ \d^- a,\ a,\ \d^+ a,\ \ldots,\ (\d^+)^{p-q}a.$$
It follows that $K$~is a totally ordered directed complex; the ordered basis is obtained from that of~$K'$ by inserting the additional elements $(\d^-)^{p-q-1}a,\ldots,(\d^+)^{p-q}a$ immediately after~$a^0$. For $r>0$ we have $(\d^+)^r a^0=(\d^+)^{p-q+r}a$; the ordered basis for~$K$ therefore finishes with the elements
$$a,\ \ldots,\ (\d^+)^{p-q}a,\ (\d^+)^{p-q+1}a,\ \ldots,\ (\d^+)^p a.$$
This completes the proof.
\end{proof}

\begin{proposition} \label{P5.4}
Let $\mathbf{s}$ be an up-down vector with last term~$p$ and let $K$ be an $\mathbf{s}$-simple chain complex with final generator~$a$. Then $\nu K$ is an $\mathbf{s}$-simple $\omega$-category with final generator~$\langle a\rangle$.
\end{proposition}

\begin{proof}
The proof is by induction on the number of terms in~$\mathbf{s}$.

Suppose that $\mathbf{s}=(p)$. According to Theorem~\ref{T4.4}, $\nu K$ has a presentation generators
$$\langle(\d^-)^p a\rangle,\ \ldots,\ \langle\d^- a\rangle,\ \langle a\rangle,\
\langle\d^+ a\rangle,\ \ldots,\ \langle(\d^+)^p a\rangle$$
and with relations
\begin{align*}
&d_{i-1}^-\langle(\d^-)^{p-i}a\rangle
=d_{i-1}^-\langle(\d^+)^{p-i}a\rangle
=\langle(\d^-)^{p-i+1}a\rangle\quad (0<i\leq p),\\
&d_{i-1}^+\langle(\d^-)^{p-i}a\rangle
=d_{i-1}^+\langle(\d^+)^{p-i}a\rangle
=\langle(\d^+)^{p-i+1}a\rangle\quad (0<i\leq p).
\end{align*}
Because of the axioms 
$$d_{i-1}^- d_i^-=d_{i-1}^- d_i^+=d_{i-1}^-,\quad
d_{i-1}^+ d_i^-=d_{i-1}^+ d_i^+=d_{i-1}^+,$$
this collapses to a presentation with a single $p$-dimensional generator~$\langle a\rangle$ and with no relations. Therefore $\nu K$ is an $\mathbf{s}$-simple $\omega$-category with final generator~$\langle a\rangle$. 

Now let $\mathbf{s}=(\mathbf{s'},q,p)$. Let $K'$, $K^0$ and $K''$ be the $\mathbf{s'}$-simple, $(q)$-simple and $(p)$-simple subcomplexes with final generators $a'$, $a^0$ and~$a$ as in Definition~\ref{D5.2}. Using the presentation of Theorem~\ref{T4.4} and the inductive hypothesis, we see that $\nu K$ is generated by $\nu K'$ and $\nu K''$ subject to the relation $d_q^+\langle a'\rangle=d_q^-\langle a\rangle$. This gives the result.
\end{proof}

We also have simple squares of chain complexes, corresponding to simpe squares of $\omega$-categories.

\begin{definition} \label{D5.5}
Let $\mathbf{s}$ be an up-down vector with more than one term given by $\mathbf{s}=(\mathbf{s'},q,p)$ and let $p'$ be the last term in~$\mathbf{s'}$; then an \emph{$\mathbf{s}$-simple square of chain complexes} is a square of augmented directed complexes
$$\xymatrix{
K^0 \ar_{\pi}[d] \ar^{\rho}[r] & K'' \ar^{\tau}[d] \\
K' \ar_{\sigma}[r] & K
}$$
with the following properties: the complexes $K'$, $K^0$ and~$K''$ are $\mathbf{s'}$-simple, $(q)$-simple and $(p)$-simple with final generators $a'$, $a^0$ and~$a''$; the morphisms $\pi$~and~$\rho$ are given by
$$\pi(\d^+)^{q-r}a^0=(\d^+)^{p'-r}a',\ \rho(\d^-)^{q-r}a^0=(\d^-)^{p-r}a''\quad
(0\leq r\leq q);$$
the square is a push-out as a square of abelian groups; the prescribed submonoid of~$K$ is the sum of the images of the prescribed submonoids of $K'$~and~$K''$.
\end{definition}

Obviously we have the following result.

\begin{proposition} \label{P5.6}
If $\tau\colon K''\to K$ is the right hand vertical morphism in an $\mathbf{s}$-simple square of chain complexes and if $K''$~has final generator~$a''$ then $K$~is an $\mathbf{s}$-simple chain complex with final generator~$a$ such that
$$\tau(\d^+)^{p-r}a''=(\d^+)^{p-r}a\quad (0\leq r\leq p).$$
\end{proposition}

Using Definition~\ref{D2.6} and Proposition~\ref{P2.8} we obtain the following results.

\begin{proposition} \label{P5.7}
The image under~$\nu$ of an $\mathbf{s}$-simple square of chain complexes is an $\mathbf{s}$-simple square of $\omega$-categories.
\end{proposition}

\begin{proposition} \label{P5.8}
Let $\omega$-$\mathbf{cat}$ be the category of $\omega$-categories, let $\Sigma$ be the category of simple chain complexes and morphisms of augmented directed complexes, and let $\hat\Sigma$ be the category of contravariant functors from~$\Sigma$ to sets which take simple squares to pull-back squares. Then there is an equivalence of categories
$$C\mapsto\Hom[\nu(-),C]\colon\textup{$\omega$-$\mathbf{cat}$}\to\hat\Sigma.$$
\end{proposition}

\section{The chain complexes of simplexes} \label{S6}

In this section we discuss the chain complexes of simplexes, which will simply be called simplexes. They correspond to the theory of sets with complicial identities (Section~{\ref{S3}). The material is mostly taken from~\cite{B5}. 

\begin{definition} \label{D6.1}
For $m=0,1,2,\ldots\,$ the \emph{$m$-simplex} $\Delta(m)$ is the free augmented directed complex constructed as follows. The basis elements correspond to the sequences of integers
$$a_0,\ldots,a_q$$
with $0\leq q\leq m$ and $0\leq a_0<a_1<\ldots<a_q\leq m$. The basis element corresponding to $a_0,\ldots,a_q$ is written $[a_0,\ldots,a_q]$ and has dimension~$q$. If $q>0$ then the boundary of $[a_0,\ldots,a_q]$ is the alternating sum
$$[a_1,\ldots,a_q]-[a_0,a_2,\ldots,a_q]+\ldots
+(-1)^q[a_0,\ldots,a_{q-1}].$$
The augmentation is given by $\e[a_0]=1$.
\end{definition}

We will now show that the simplexes are totally ordered directed complexes by expressing them as joins.

\begin{definition} \label{D6.2}
Let $K$~and~$L$ be augmented directed complexes. Then the \emph{join} $K*L$ is the the direct sum of abelian groups
$$K*L=K\oplus(K\times L)\oplus L$$
with the following structure. The grading is given by
$$(K*L)_q=K_q\oplus\left[\bigoplus_{i+j=q-1}(K_i\otimes L_j)\right]\oplus L_q.$$
The inclusions of $K$~and~$L$ in $K*L$ commute with the boundary and augmentation homomorphisms. The boundary on $K_i\otimes L_j$ is given by
$$\d(x\otimes y)=\begin{cases}
(\e x)y-(\e y)x& (i=j=0),\\
(\e x)y-x\otimes\d y& (i=0,\ j>0),\\
\d x\otimes y-(-1)^i(\e y)x& (i>0,\ j=0),\\
\d x\otimes y-(-1)^i x\otimes\d y& (i,j>0).
\end{cases}$$
The prescribed submonoid of $K*L$ is generated by the elements of the prescribed submonoids of $K$~and~$L$ and by the tensor products of these elements.
\end{definition}

\begin{example} \label{E6.3}
The $m$-simplex $\Delta(m)$ is the join of $m+1$ copies of $\Delta(0)$.
\end{example}

\begin{proposition} \label{P6.4}
If $K$~and~$L$ are totally ordered directed complexes, then $K*L$ is a totally ordered directed complex.
\end{proposition}

\begin{proof}
One can check that $K*L$ has a suitably ordered basis consisting of the basis elements of $K$~and~$L$ and of the tensor products of these basis elements. The ordering of the basis for $K*L$ is obtained as follows. Take the basis elements of~$K$ in order followed by the basis elements of~$L$ in order. If $a$~is an odd-dimensional basis element in~$K$, then the basis elements of the form $a\otimes b$ are inserted before~$a$ in the order given by the second factor; if $a$~is an even-dimensional basis element in~$K$, then the basis elements of the form $a\otimes b$ are inserted after~$a$ in the reverse of the order given by the second factor.
\end{proof}

\begin{proposition} \label{P6.5}
A simplex is a totally ordered directed complex.
\end{proposition}

\begin{proof}
Obviously $\Delta(0)$ is a totally ordered directed complex. The result now  follows from Example~\ref{E6.3} and Proposition~\ref{P6.4}.
\end{proof}

We will now use simplexes to construct a functor~$\lambda$ from augmented directed complexes to sets with complicial identities. The $m$-dimensional elements in $\lambda K$ will be the morphisms of augmented directed complexes from $\Delta(m)$ to~$K$. An operation~$\theta$ in sets with complicial identities will be contravariantly represented by a morphism~$\theta^\v$ between simplexes. In particular there are the obvious morphisms corresponding to the face and degeneracy operations.

\begin{notation} \label{N6.6}
The face and degeneracy morphisms
\begin{align*}
&\d_i^\v\colon\Delta(m-1)\to\Delta(m)\quad (m>0,\ 0\leq i\leq m),\\
&\e_i^\v\colon\Delta(m+1)\to\Delta(m)\quad (0\leq i\leq m)
\end{align*}
are defined on basis elements as follows. 

If $a=[a_0,\ldots,a_q]$ is a basis element for $\Delta(m-1)$ then $\d_i^\v a=[a'_0,\ldots,a'_q]$ with
$$a'_q=\begin{cases}a_q&(0\leq a_q\leq i-1),\\ a_q+1&(i\leq a_q\leq m-1).\end{cases}$$

If $b=[b_0,\ldots,b_q]$ is a basis element for $\Delta(m+1)$ including both the terms $i$~and $i+1$ then $\e_i^\v b=0$.

If $b=[b_0,\ldots,b_q]$ is a basis element for $\Delta(m+1)$ not including both the terms $i$~and $i+1$ then $\e_i^\v b=[b''_0,\ldots,b''_q]$ with
$$b''_q=\begin{cases}b_q&(0\leq b_q\leq i),\\ b_q-1&(i+1\leq b_q\leq m+1).\end{cases}$$
\end{notation}

Recall from Definition~\ref{D4.1} that morphisms of augmented directed complexes are augmentation-preserving chain maps taking prescribed submonoids into prescribed submonoids. Recall also that the prescribed submonoid of a free augmented directed complex is the submonoid generated by the prescribed basis elements. We obviously have the following result.

\begin{proposition} \label{P6.7}
The face and degeneracy morphisms are morphisms of augmented directed complexes.
\end{proposition}

Less obviously we also have the following result.

\begin{proposition} \label{P6.8}
Let $K$ be an augmented directed complex and let
$$x,y\colon\Delta(m)\to K$$
be morphisms of augmented directed complexes such that $x\d_i^\v=y\d_{i+1}^\v$ for some~$i$ with $0\leq i<m$. Then there is a morphism of augmented directed complexes 
$$z\colon\Delta(m+1)\to K$$
given by
$$z
=x\e_{i+1}^\v-x\d_i^\v(\e_i^\v)^2+y\e_i^\v
=x\e_{i+1}^\v-y\d_{i+1}^\v(\e_i^\v)^2+y\e_i^\v.$$
\end{proposition}

\begin{proof}
It is clear that $z$~is an augmentation preserving chain map; it therefore suffices to prove that $za$ is in the prescribed submonoid of~$K$ for each basis element~$a$ in $\Delta(m+1)$. We do this by considering three cases: if $a$~has no term $i+2$ then
$\d_{i+1}^\v(\e_i^\v)^2 a=\e_i^\v a$, hence $za=x\e_{i+1}^\v a$; if $a$~has no term~$i$ then $\d_i^\v(\e_i^\v)^2 a=\e_{i+1}^\v a$, hence $za=y\e_i^\v a$; if $a$~has terms $i$~and~$i+2$ then $(\e_i^\v)^2 a=0$, hence 
$$za=x\e_{i+1}^\v a+y\e_i^\v a.$$
\end{proof}

The definition of~$\lambda$ is now as follows.

\begin{notation} \label{N6.9}
Let $K$ be an augmented directed complex. Then $\lambda K$ is the graded set given by
$$(\lambda K)_m=\Hom[\Delta(m),K].$$
If $x\in(\lambda K)_m$ with $m>0$ and if $0\leq i\leq m$ then
$$\d_i x=x\d_i^\v.$$
If $x\in(\lambda K)_m$ and $0\leq i\leq m$ then
$$\e_i x=x\e_i^\v.$$
If $x,y\in(\lambda K)_m$ and $\d_i x=\d_{i+1}y$ for some~$i$ with $0\leq i<m$ then
$$x\w_i y=\e_{i+1}x-\e_i^2 \d_i x+\e_i y =\e_{i+1}x-\e_i^2 \d_{i+1}y+\e_i y.$$
\end{notation}

\begin{proposition} \label{P6.10}
If $K$~is an augmented directed complex then $\lambda K$ is a set with complicial identities.
\end{proposition}

\begin{proof}
We see that the operations are well-defined. The axioms follow straightforwardly from computations with chain maps.
\end{proof}

We conclude this section with the main result of~\cite{B5} (Theorem~8.7). 

\begin{theorem} \label{T6.11}
Let $\Or$ be the full subcategory of the category of augmented directed complexes with objects $\Delta(0)$, $\Delta(1)$,~\dots. For $n\geq 0$ let $\iota_n$ be the identity endomorphism of $\Delta(n)$. Then $\lambda$~is a fully faithful embedding of~$\Or$ in the category of sets with complicial identities such that $\lambda\Delta(n)$ is freely generated by the $n$-dimensional element $\lambda\iota_n$.
\end{theorem}

\section{Complicial identities in terms of chain complexes} \label{S7}

In the last section we constructed a functor~$\lambda$ from augmented directed complexes to sets with complicial identities (see Proposition \ref{P6.10}). A set with complicial identities~$X$ therefore defines a contravariant set-valued functor 
$$K\mapsto\Hom(\lambda K,X)$$
on the category~$\textbf{adc}$ of augmented directed complexes. We will now reverse this process: we will show that sets with complicial identities can be obtained from   contravariant set-valued functors on a suitable subcategory of $\textbf{adc}$, provided that they take certain diagrams to limit diagrams.

The objects and diagrams correspond to the domains in the axioms for sets with complicial identities (see Definition~\ref{D3.1}), and we will now consider the various axioms.

The augmented directed complexes associated to axioms (1)~and~(3) are the simplexes $\Delta(m)$.

In the remaining cases we use diagrams of augmented directed complexes which are colimit diagrams as diagrams of abelian groups. The prescribed submonoid of the target object is always the sum of the images of the prescribed submonoids of the other objects in the diagram.

For axioms (2)~and~(4) we use diagrams 
$$\xymatrix{
\Delta(m-1) \ar_{\d_i^\v}[d] \ar^{\d_{i+1}^\v}[r] & \Delta(m) \ar^{\eta_y}[d] \\
\Delta(m) \ar_{\eta_x}[r] & \Delta_{(2)}(m,i)
}$$
with $0\leq i<m$. Since these diagrams are to be colimit diagrams as diagrams of abelian groups, we have
$$\Delta_{(2)}(m,i)
\cong\dfrac{\Delta(m)\oplus\Delta(m)}
{\{\,(\d_i^\v z,-\d_{i+1}^\v z):z\in\Delta(m-1)\,\}}.$$
If $X$~is a set with complicial identities then $\Delta_{(2)}(m,i)$ corresponds to the limit
$$\{\,(x,y)\in X_m\times X_m:\d_i x=\d_{i+1}y\,\}.$$
We will also need the morphisms
$$v_i^\v\colon\Delta(m+1)\to\Delta_{(2)}(m,i)$$
corresponding to the wedge operations; these are given by
$$v_i^\v
=\eta_x\e_{i+1}^\v-\eta_x\d_i^\v(\e_i^\v)^2+\eta_y\e_i^\v
=\eta_x\e_{i+1}^\v-\eta_y\d_{i+1}^\v(\e_i^\v)^2+\eta_y\e_i^\v.$$

For axiom~(5) we use similar diagrams 
$$\xymatrix{
\Delta(m) \ar_{\d_i^\v}[dddd] \ar^{v_i^\v\d_{i+1}^\v}[r] 
& \Delta_{(2)}(m,i) \ar[dddd] \\
\\
&& \Delta(m) \ar[uul] \ar_{\eta_y}[ddl]
& \Delta(m-1) \ar_{\d_i^\v}[l] \ar^{\d_{i+1}^\v}[r]
& \Delta(m) \ar[uulll] \ar^{\eta_z}[ddlll] \\
\\
\Delta(m+1) \ar_{\eta_b}[r]
& \Delta_{(5)}(m,i)
}$$
with $0\leq i<m$.

For axiom~(6) we use diagrams 
$$\xymatrix{
&&& \Delta_{(2)}(m,i) \ar[dddd] 
& \Delta(m) \ar^{v_i^\v\d_{i+1}^\v}[l] \ar_{\d_{i+2}^\v}[dddd] \\
\\
\Delta(m) \ar[uurrr] \ar_{\eta_x}[ddrrr]
& \Delta(m-1) \ar_{\d_i^\v}[l] \ar^{\d_{i+1}^\v}[r]
& \Delta(m) \ar[uur] \ar^{\eta_y}[ddr] \\
\\
&&& \Delta_{(6)}(m,i) 
& \Delta(m+1) \ar^{\eta_c}[l]
}$$
with $0\leq i<m$.

For axiom~(7) we use diagrams 
$$\xymatrix{
\Delta(m-1) \ar_{\d_i^\v}[d] \ar^{\d_{i+1}^\v}[r] &
\Delta(m) \ar_{\eta_y}[d] &
\Delta(m-1) \ar_{\d_i^\v}[l] \ar^{\d_{i+1}^\v}[d] \\
\Delta(m) \ar_{\eta_x}[r] &
\Delta_{(7)}(m,i) &
\Delta(m) \ar^{\eta_z}[l]
}$$
with $0\leq i<m$.

For axiom~(8) we use diagrams
$$\xymatrix{
\Delta(m-1) \ar_{\d_i^\v}[ddd] \ar[r] 
& \Delta_{(2)}(m-1,1) 
& \ar[l] \Delta(m-1) \ar^{\d_{i+2}^\v}[ddd]\\
& \Delta(m-1) \ar^{v_i^\v\d_{i+1}^\v}[u] \ar_{\d_{i+1}^\v}[d]\\
& \Delta(m) \ar_{\eta_w}[d]\\
\Delta(m) \ar^{\eta_x}[r] & \Delta_{(7)}(m,i) & \ar_{\eta_z}[l] \Delta(m)\\
\Delta(m-1) \ar^{\d_{i+1}^\v}[u] \ar_{\d_{i+2}^\v}[r] 
& \Delta(m) \ar^{\eta_y}[u] 
& \ar^{\d_i^\v}[l] \Delta(m-1) \ar_{\d_{i+1}^\v}[u]
}$$
with $0\leq i\leq m-2$.

For axiom~(9) we use diagrams
$$\xymatrix{
\Delta(m-1) \ar^{\d_{i+1}^\v}[r] \ar_{\d_i^\v}[d]
& \Delta(m) \ar_{\eta_y}[d]
& \Delta(m-1) \ar_{\d_{j-1}^\v}[l] \ar^{\d_j^\v}[d] \\
\Delta(m) \ar^{\eta_x}[r]
& \Delta_{(9)}(m,i,j)
& \Delta(m) \ar_{\eta_w}[l] \\
\Delta(m-1) \ar_{\d_j^\v}[r] \ar^{\d_{j-1}^\v}[u]
& \Delta(m) \ar^{\eta_z}[u]
& \Delta(m-1) \ar^{\d_i^\v}[l] \ar_{\d_{i+1}^\v}[u]
}$$
with $0\leq i\leq j-3\leq m-3$.

We can evidently obtain sets with complicial identities from contravariant set-valued functors in the following way.

\begin{proposition} \label{P7.1}
Let $\Pi$ be the full subcategory of the category of augmented directed complexes given by the objects in the diagrams associated to the axioms for sets with complicial identities. Let $X$ be a contravariant set-valued functor on~$\Pi$ taking each of the diagrams to a limit diagram. Then there is a set with complicial identities functorial in~$X$ such that the $m$-dimensional elements are the members of $X[\Delta(m)]$ and such that the operations are induced by the morphisms $\d_i^\v$, $\e_i^\v$ and~$v_i^\v$.
\end{proposition}

In particular let $C$ be an $\omega$-category; then there is a contravariant set-valued functor on the category~$\Pi$ of this definition given by $\Hom[\nu(-),C]$. We want this functor to yield a set with complicial identities. In order to do this, we must show that the images under~$\nu$ of the diagrams associated to the axioms are colimit diagrams of $\omega$-categories. We will do this by showing that the objects of~$\Pi$ are totally ordered directed complexes; the colimit properties will then be consequences of the presentations in terms of atoms (Theorem~\ref{T4.4}).

\begin{proposition} \label{P7.2}
If $K$~is an object in a diagrams associated to an axiom for sets with complicial identities, then $K$~is a totally ordered directed complex.
\end{proposition}

\begin{proof}
We already know from Proposition~\ref{P6.5} that the simplexes $\Delta(m)$ are totally ordered directed complexes, because they are joins of copies of $\Delta(0)$ and because $\Delta(0)$ is a totally ordered directed complex. We will prove the result for the other complexes involved in a similar way, by expressing them as joins. It is convenient to write $\Delta(-1)$ for the zero chain complex, which serves as an identity for the join construction; it is then straightforward to verify that
\begin{align*}
&\Delta_{(k)}(m,i)\cong\Delta(i-1)*\Delta_{(k)}(1,0)*\Delta(m-i-2)\quad (k=2,5,6,7),\\
&\Delta_{(8)}(m,i)\cong\Delta(i-1)*\Delta_{(8)}(2,0)*\Delta(m-i-3),\\
&\Delta_{(9)}(m,i,j)\\
&\qquad{}\cong\Delta(i-1)*\Delta_{(2)}(1,0)*\Delta(j-i-4)*\Delta_{(2)}(1,0)*\Delta(m-j-1).
\end{align*}
It now suffices to show that $\Delta_{(k)}(1,0)$ is a totally ordered directed complex for $k=2,5,6,7$ and that $\Delta_{(8)}(2,0)$ is a totally ordered directed complex. We will do this in each case by drawing a figure and listing the basis elements in the correct order.

For $\Delta_{(2)}(1,0)$ the figure is
$$\begin{xy}
0;<1cm,0cm>:
(0,0) *{\bullet},
(2,0) *{\bullet},
(4,0) *{\bullet},
(0,0);(1,0) **@{-} ?>*\dir{>}, (1,0);(2,0) **@{-},
(2,0);(3,0) **@{-} ?>*\dir{>}, (3,0);(4,0) **@{-},
(1,0) *+!D{\eta_x},
(3,0) *+!D{\eta_y}
\end{xy}$$
and the ordered basis is
$$\eta_x[0],\ \eta_x[0,1],\ \eta_x[1]=\eta_y[0],\ \eta_y[0,1],\ \eta_y[1].$$

For $\Delta_{(5)}(1,0)$ the figure is
$$
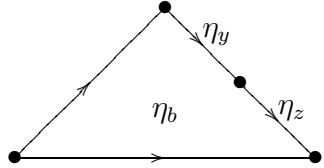
\begin{xy}
0;<1cm,0cm>:
(0,0) *{\bullet},
(2,2) *{\bullet},
(3,1) *{\bullet},
(4,0) *{\bullet},
(0,0);(1,1) **@{-} ?>*\dir{>}, (1,1);(2,2) **@{-},
(2,2);(2.5,1.5) **@{-} ?>*\dir{>}, (2.5,1.5);(3.5,0.5) **@{-} ?>*\dir{>}, (3.5,0.5);(4,0) **@{-},
(0,0);(2,0) **@{-} ?>*\dir{>}, (2,0);(4,0) **@{-},
(2,0.6) *{\eta_b},
(2.5,1.5) *!DL{\eta_y},
(3.5,0.5) *!DL{\eta_z}
\end{xy}$$
and the ordered basis is
\begin{align*}
&\eta_b[0],\ \eta_b[0,2],\ \eta_b[0,1,2],\ \eta_b[0,1],\ \eta_b[1]=\eta_y[0],\\ 
&\eta_y[0,1],\ \eta_y[1]=\eta_z[0],\ \eta_z[0,1],\ \eta_z[1],
\end{align*}
with
$$\d^+\eta_b[0,1,2]=\eta_b[0,1]+\eta_b[1,2]=\eta_b[0,1]+\eta_y[0,1]+\eta_z[0,1].$$

For $\Delta_{(6)}(1,0)$ the figure is
$$
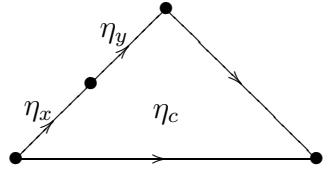
\begin{xy}
0;<1cm,0cm>:
(0,0) *{\bullet},
(1,1) *{\bullet},
(2,2) *{\bullet},
(4,0) *{\bullet},
(0,0);(0.5,0.5) **@{-} ?>*\dir{>}, (0.5,0.5);(1.5,1.5) **@{-} ?>*\dir{>}, (1.5,1.5);(2,2) **@{-},
(2,2);(3,1) **@{-} ?>*\dir{>}, (3,1);(4,0) **@{-},
(0,0);(2,0) **@{-} ?>*\dir{>}, (2,0);(4,0) **@{-},
(0.5,0.5) *!DR{\eta_x},
(1.5,1.5) *!DR{\eta_y},
(2,0.6) *{\eta_c}
\end{xy}$$
and the ordered basis is
\begin{align*}
&\eta_c[0],\ \eta_c[0,2],\ \eta_c[0,1,2],\ \eta_x[0,1],\ \eta_x[1]=\eta_y[0],\ \eta_y[0,1],\\ 
&\eta_y[1]=\eta_c[1],\ \eta_c[1,2],\ \eta_c[2],
\end{align*}
with
$$\d^+\eta_c[0,1,2]=\eta_c[0,1]+\eta_c[1,2]=\eta_x[0,1]+\eta_y[0,1]+\eta_c[1,2].$$

For $\Delta_{(7)}(1,0)$ the figure is
$$\begin{xy}
0;<1cm,0cm>:
(0,0) *{\bullet},
(2,0) *{\bullet},
(4,0) *{\bullet},
(6,0) *{\bullet},
(0,0);(1,0) **@{-} ?>*\dir{>}, (1,0);(2,0) **@{-},
(2,0);(3,0) **@{-} ?>*\dir{>}, (3,0);(4,0) **@{-},
(4,0);(5,0) **@{-} ?>*\dir{>}, (5,0);(6,0) **@{-},
(1,0) *+!D{\eta_x},
(3,0) *+!D{\eta_y},
(5,0) *+!D{\eta_z}
\end{xy}$$
and the ordered basis is
$$\eta_x[0],\ \eta_x[0,1],\ \eta_x[1]=\eta_y[0],\ \eta_y[0,1],\ \eta_y[1]=\eta_z[0],\ \eta_z[0,1],\ \eta_z[1].$$

For $\Delta_{(8)}(2,0)$ the figure is
$$
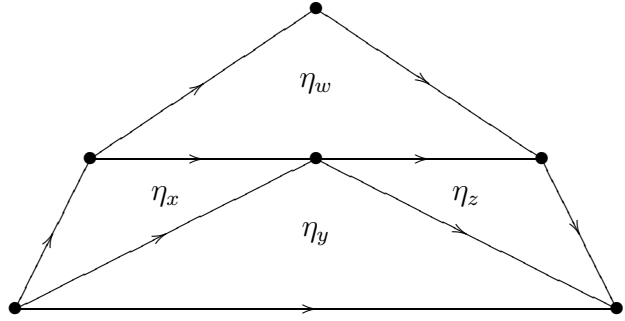
\begin{xy}
0;<1cm,0cm>:
(0,0) *{\bullet},
(1,2) *{\bullet},
(4,2) *{\bullet},
(4,4) *{\bullet},
(7,2) *{\bullet},
(8,0) *{\bullet},
(0,0);(0.5,1) **@{-} ?>*\dir{>}, (0.5,1);(1,2) **@{-},
(1,2);(2.5,2) **@{-} ?>*\dir{>}, (2.5,2);(5.5,2) **@{-} ?>*\dir{>}, (5.5,2);(7,2) **@{-},
(7,2);(7.5,1) **@{-} ?>*\dir{>}, (7.5,1);(8,0) **@{-},
(0,0);(4,0) **@{-} ?>*\dir{>}, (4,0);(8,0) **@{-},
(0,0);(2,1) **@{-} ?>*\dir{>}, (2,1);(4,2) **@{-},
(4,2);(6,1) **@{-} ?>*\dir{>}, (6,1);(8,0) **@{-},
(1,2);(2.5,3) **@{-} ?>*\dir{>}, (2.5,3);(4,4) **@{-},
(4,4);(5.5,3) **@{-} ?>*\dir{>}, (5.5,3);(7,2) **@{-},
(2,1.5) *{\eta_x},
(4,1) *{\eta_y},
(6,1.5) *{\eta_z},
(4,3) *{\eta_w}
\end{xy}$$
and the ordered basis is
\begin{align*}
&\eta_y[0],\ \eta_y[0,2],\ \eta_y[0,1,2],\ \eta_y[0,1]=\eta_x[0,2],\ \eta_x[0,1,2],\ \eta_x[0,1],\\
&\eta_x[1],\ 
\eta_x[1,2],\ \eta_x[2]=\eta_z[0],\
\eta_z[0,2],\ \eta_z[0,1,2],\ \eta_z[0,1],\\ 
&\eta_w[0,1,2],\ \eta_w[0,1],\ \eta_w[1],\ \eta_w[1,2],\ \eta_w[2]=\eta_z[1],\ \eta_z[1,2],\ \eta_z[2],
\end{align*}
with
$$\d^-\eta_w[0,1,2]=\eta_w[0,2]=\eta_x[1,2]+\eta_z[0,1].$$

This completes the proof.
\end{proof}

Now let $\Pi$ be the category of Proposition~\ref{P7.1} and let $C$ be an $\omega$-category. We have shown in Proposition~\ref{P7.2} that the objects of~$\Pi$ are totally ordered directed complexes. It follows from the atomic presentations (Theorem~\ref{T4.4}) that the images under~$\nu$ of the diagrams of Proposition~\ref{P7.1} are colimit diagrams of $\omega$-categories. It therefore follows from Proposition~\ref{P7.1} that one obtains a set with complicial identities from the functor 
$$K\mapsto\Hom[\nu K,C].$$
We will use the following notation.

\begin{notation} \label{N7.3}
Let $\alpha$ be the functor from $\omega$-categories to sets with complicial identities defined on an $\omega$-category~$C$ as follows. The set of $m$-dimensional elements is given by
$$(\alpha C)_m=\Hom[\nu\Delta(m),C].$$
If $x\in(\alpha C)_m$ then
\begin{align*}
&\d_i x=x(\nu\d_i^\v)\quad (m>0,\ 0\leq i\leq m),\\
&\e_i x=x(\nu\e_i^\v)\quad (0\leq i\leq m).
\end{align*}
If $x,y\in(\alpha C)_m$ and $\d_i x=\d_{i+1}y$ with $0\leq i<m$ then 
$$x\w_i y=z v_i^\v,$$
where $z$~is the member of $\Hom[\nu\Delta_{(2)}(m,i),C]$ with
$$z(\nu\eta_x)=x,\ z(\nu\eta_y)=y.$$
\end{notation}

\section{Simple chain complexes as quotients of simplexes} \label{S8}

At the end of Section~\ref{S7} we have constructed a functor~$\alpha$ from $\omega$-categories to sets with complicial identities. We also need a functor in the opposite direction. Equivalently (Proposition~\ref{P5.8}), given sets with complicial identities, we need contravariant set-valued functors with suitable properties on the category of simple chain complexes. We will again use the functor~$\lambda$ of Proposition~\ref{P6.10} from augmented directed complexes to sets with complicial identities; the functor on simple chain complexes corresponding to a set with complicial identities~$X$ will be given by
$$S\mapsto\Hom[\lambda S,X].$$

We must show that these functors take simple squares of chain complexes to pull-backs (see Proposition~\ref{P5.8}). We will obtain information about the sets $\Hom[\lambda S,X]$ by showing that simple chain complexes are retracts of simplexes. In this section, as a first step, we show that an $\mathbf{s}$-simple chain complex can be expressed as a quotient
$$S_\mathbf{s}=\Delta(|\mathbf{s}|)/U_\mathbf{s},$$
where $\Delta(|\mathbf{s}|)$ is a simplex of a suitable dimension. We will also show an $\mathbf{s}$-simple square of chain complexes can be obtained from a commutative square of simplexes~$Q_\mathbf{s}$.

We will now define our notations. 

\begin{notation} \label{N8.1}
Let $\mathbf{s}$ be an up-down vector given by
$$\mathbf{s}=(p_0,q_1,p_1,\ldots,p_{k-1},q_k,p_k).$$
Then
$$|\mathbf{s}|=p_0-q_1+p_1-\ldots+p_{k-1}-q_k+p_k.$$
\end{notation}

We make the following observation, which will be used frequently in inductive arguments, mostly without comment.

\begin{proposition} \label{P8.2}
Let $\mathbf{s}$ be an up-down vector with more than one term given by $\mathbf{s}=(\mathbf{s'},q,p)$, and let $p'$ be the last term in~$\mathbf{s'}$. Then
$$|\mathbf{s}|-p=|\mathbf{s'}|-q>|\mathbf{s'}|-p'.$$
\end{proposition}

\begin{proof}
This holds because $|\mathbf{s}|=|\mathbf{s'}|-q+p$ and because $p'>q$.
\end{proof}

\begin{notation} \label{N8.3}
Let $\mathbf{s}=(\mathbf{s'},q,p)$ be an up-down vector with more than one term. Then $Q_\mathbf{s}$~is the commutative square
$$\xymatrix{
\Delta(q) \ar_{(\d_0^\v)^{|\mathbf{s}|-p}}[d] \ar^{(\d_1^\v)^{p-q}}[rr] && 
\Delta(p) \ar^{(\d_0^\v)^{|\mathbf{s}|-p}}[d] \\
\Delta(|\mathbf{s'}|) \ar_{(\d_{|\mathbf{s}|-p+1}^\v)^{p-q}}[rr] && \Delta(|\mathbf{s}|).
}$$
\end{notation}

\begin{notation} \label{N8.4}
For $p\geq 0$ let $U_{(p)}$ be the subcomplex of $\Delta(p)$ generated by the basis elements $[i_0,\ldots,i_m]$ with $m>0$ and $i_1\leq p-m$. 

For $\mathbf{s}=(\mathbf{s'},q,p)$ let $V_\mathbf{s}$ be the subcomplex of $\Delta(|\mathbf{s}|)$ generated by the basis elements $[i_0,\ldots,i_{r-1},|\mathbf{s}|-p,i_{r+1},\ldots,i_m]$ with $0<r<m$ and
$$0\leq i_{r-1}<|\mathbf{s}|-p<i_{r+1}\leq |\mathbf{s}|-q,$$
and let
$$U_\mathbf{s}
=(\d_{|\mathbf{s}|-p+1}^\v)^{p-q}U_\mathbf{s'}
+(\d_0^\v)^{|\mathbf{s}|-p}U_{(p)}
+V_\mathbf{s}.$$
\end{notation}

\begin{notation} \label{N8.5}
For an arbitrary up-down vector~$\mathbf{s}$, let
$$S_\mathbf{s}=\Delta(|\mathbf{s}|)/U_\mathbf{s}.$$
\end{notation}

\begin{remark} \label{R8.6}
As an abelian group, $U_\mathbf{s}$ is generated by basis elements of positive dimension and by their boundaries. It follows that the quotient~$S_\mathbf{s}$ is naturally an augmented chain complex and the quotient homomorphism
$$\Delta(|\mathbf{s}|)\to\Delta(|\mathbf{s}|)/U_\mathbf{s}=S_\mathbf{s}$$
is augmentation-preserving. We make~$S_\mathbf{s}$ into an augmented directed complex by taking the images of the basis elements for $\Delta(|\mathbf{s}|)$ as generators for the prescribed submonoid of~$S_\mathbf{s}$. This makes
the quotient homomorphism into a morphism of augmented directed complexes.
\end{remark}

We will now consider the one-term case.

\begin{proposition} \label{P8.7}
Let $p$ be a nonnegative integer and let
$$a_m^i=[i,p-m+1,p-m+2,\ldots,p]\quad (0\leq i\leq p-m\leq p).$$
Then $S_{(p)}$~is a $(p)$-simple chain complex with generator~$a$ such that
\begin{align*}
&a_m^i+U_{(p)}=(\d^-)^{p-m}a\quad (0\leq i<p-m\leq p),\\
&a_m^{p-m}+U_{(p)}=(\d^+)^{p-m}a\quad (0\leq m\leq p).
\end{align*}
\end{proposition}

\begin{proof}
We use Definition~\ref{D5.1}. Note that $U_{(p)}$~is the subcomplex of $\Delta(p)$ generated by the basis elements not of the form~$a_m^i$. If $b$~is a generator of~$U_{(p)}$ of the form
$$b=[j,i,p-m+1,p-m+2,p-m+3,\ldots,p]\quad (0\leq j<i<p-m\leq p),$$
then
$$\d b=a_m^i-a_m^j+u$$
with $u\in U_{(p)}$; if $b$~is any other generator for~$U_{(p)}$ then $\d b\in U_{(p)}$. As an abelian group, $U_{(p)}$~is therefore generated by the basis elements not of the form~$a_m^i$ and by the differences
$$a_m^i-a_m^0\quad (0<i<p-m\leq p).$$
It follows that $S_{(p)}$~is a free augmented directed complex with basis
$$a_0^0+U_{(p)},\ \ldots,\ a_{p-1}^0+U_{(p)},\ a_p^0+U_{(p)},\ 
a_{p-1}^1+U_{(p)},\ \ldots,\ a_0^p+U_{(p)}$$
and that
$$a_m^i+U_{(p)}=a_m^0+U_{(p)}\quad (0\leq i<p-m\leq p).$$
It is straightforward to check that
\begin{align*}
&\d(a_{m+1}^0+U_{(p)})=a_m^{p-m}-a_m^0+U_{(p)}\quad (0\leq m<p),\\
&\d(a_{m+1}^{p-m-1}+U_{(p)})=a_m^{p-m}-a_m^{p-m-1}+U_{(p)}\quad (0\leq m<p),
\end{align*}
from which it follows that
\begin{align*}
&a_m^i+U_{(p)}=(\d^-)^{p-m}(a_p^0+U_{(p)})\quad (0\leq i<p-m\leq p),\\
&a_m^{p-m}+U_{(p)}=(\d^+)^{p-m}(a_p^0+U_{(p)})\quad (0\leq m\leq p).
\end{align*}
We also have 
$$\e(a_0^0+U_{(p)})=\e(a_0^p+U_{(p)})=1;$$
therefore $S_{(p)}$~is $(p)$-simple with generator $a_p^p+U_{(p)}$, and the images $a_m^i+U_{(p)}$ are as described.
\end{proof}

We will now consider up-down vectors with more than one term. We first show that we can pass to quotients in the squares~$Q_\mathbf{s}$.

\begin{proposition} \label{P8.8}
Let $\mathbf{s}=(\mathbf{s'},q,p)$. Then the morphisms in the square~$Q_\mathbf{s}$ restrict to morphisms between the subcomplexes $U_\mathbf{s'}$, $U_{(q)}$, $U_{(p)}$,~$U_\mathbf{s}$.
\end{proposition}

\begin{proof}
We have 
$$(\d_{|\mathbf{s}|-p+1}^\v)^{p-q}U_\mathbf{s'}\subset U_\mathbf{s},\quad
(\d_0^\v)^{|\mathbf{s}|-p}U_{(p)}\subset U_\mathbf{s}$$
by definition. We also have $(\d_1^\v)^{p-q}U_{(q)}\subset U_{(p)}$ by considering generators. It therefore remains to show that $(\d_0^\v)^{|\mathbf{s}|-p}U_{(q)}\subset U_\mathbf{s'}$. To do this, let $p'$ be the last term of~$\mathbf{s'}$, so that
$$(\d_0^\v)^{|\mathbf{s}|-p}U_{(q)}
=(\d_0^\v)^{|\mathbf{s'}|-q}U_{(q)}
=(\d_0^\v)^{|\mathbf{s'}|-p'}(\d_0^\v)^{p'-q}U_{(q)}.$$
We have $(\d_0^\v)^{p'-q}U_{(q)}\subset U_{(p')}$ by considering generators. We also have
$$(\d_0^\v)^{|\mathbf{s'}|-p'}U_{(p')}\subset U_\mathbf{s'}$$
trivially (if $\mathbf{s'}=(p')$) or by definition (if $\mathbf{s'}$~has more than one term). Therefore $(\d_0^\v)^{|\mathbf{s}|-p}U_{(q)}\subset U_\mathbf{s'}$ as required.
\end{proof}

It therefore makes sense to use the following notation.

\begin{notation} \label{N8.9}
Let $\mathbf{s}=(\mathbf{s'},q,p)$ be an up-down vector with more than one term. Then $R_\mathbf{s}$~is the commutative square
$$\xymatrix{
S_{(q)} \ar[d] \ar[r] & S_{(p)} \ar[d] \\
S_\mathbf{s'} \ar[r] & S_\mathbf{s}
}$$
induced by~$Q_\mathbf{s}$.
\end{notation}

We want to show that these squares are simple in the sense of Definition~\ref{D5.5}. In particular we want to show that they are push-outs as squares of abelian groups, and we begin with the following computations. 

\begin{proposition} \label{P8.10}
Let $\mathbf{s}=(\mathbf{s'},q,p)$ and let $T'$, $T^0$,~$T''$ be the subcomplexes of $\Delta(|\mathbf{s}|)$ given by
\begin{align*}
&T'=(\d_{|\mathbf{s}|-p+1}^\v)^{p-q}\Delta(|\mathbf{s'}|),\\
&T^0
=(\d_{|\mathbf{s}|-p+1}^\v)^{p-q})(\d_0^\v)^{|\mathbf{s}|-p}\Delta(q)
=(\d_0^\v)^{|\mathbf{s}|-p}(\d_1^\v)^{p-q})\Delta(q),\\
&T''=(\d_0^\v)^{|\mathbf{s}|-p}\Delta(p).
\end{align*}
Then
$$T'\cap T''=T^0,\quad \Delta(|\mathbf{s}|)=(T'+T'')\oplus V_\mathbf{s},$$
and every standard basis element for $\Delta(|\mathbf{s}|)$ is congruent modulo~$V_\mathbf{s}$ to a sum of basis elements in $T'+T''$.
\end{proposition}

\begin{proof}
We consider various sets of basis elements of $\Delta(|\mathbf{s}|)$. Let
$$J'=\{0,1,\ldots,|\mathbf{s}|-p-1\},\quad 
J''=\{|\mathbf{s}|-p+1,|\mathbf{s}|-p+2,\ldots,|\mathbf{s}|-q\},$$
let $A'$ be the set of basis elements with no terms in~$J''$, and let $A''$ be the set of basis elements with no terms in~$J'$. We see that $T'$, $T''$ and~$T^0$ have bases $A'$, $A''$ and $A'\cap A''$ respectively; therefore $T'\cap T''=T^0$. 

Further, let $B$ be the set of generators for~$V_\mathbf{s}$; that is, $B$~is the set of basis elements containing terms in both $J'$~and~$J''$ and also containing a term $|\mathbf{s}|-p$. We see that $A'\cup A''$ and~$B$ are disjoint; we also see that the boundary of a member of~$B$ has exactly one term not in $A'\cup A''\cup B$, and that each basis element not in $A'\cup A''\cup B$ arises in this way from exactly one member of~$B$; therefore
$\Delta(|\mathbf{s}|)=(T'+T'')\oplus V_\mathbf{s}$.

It now suffices to show that every basis element~$c$ not in $A'\cup A''\cup B$ is congruent to a sum of members of $A'\cup A''$ modulo~$V_\mathbf{s}$. To do this, note that $c$~has terms in both $J'$~and~$J''$ but has no term $|\mathbf{s}|-p$. Let $b$ be the member of~$B$ obtained from~$c$ by inserting $|\mathbf{s}|-p$, so that $c$~is a term in $\d b$, and let $u',u''$ be the terms of $\d b$ adjacent to~$c$. We see that
$$\d b=\pm (u'-c+u'')+v$$
such that $v$~is a linear combination of members of~$B$; therefore $c$~is congruent to $u'+u''$ modulo~$V_\mathbf{s}$. We also see that $u'$~and~$u''$ are in $A'\cup A''\cup B$; therefore $c$~is congruent modulo~$V_\mathbf{s}$ to a sum of basis elements in $T'+T''$ as required.
\end{proof}

\begin{proposition} \label{P8.11}
Let $\mathbf{s}=(\mathbf{s'},q,p)$ be an up-down vector with more than one term. Then the square~$R_\mathbf{s}$ is a push-out as a square of abelian groups. The prescribed submonoid in the target object~$S_\mathbf{s}$ is generated by the images of the prescribed submonoids in $S_\mathbf{s'}$~and~$S_{(p)}$.
\end{proposition}

\begin{proof}
Recall that
$$U_\mathbf{s}
=(\d_{|\mathbf{s}|-p+1}^\v)^{p-q}U_\mathbf{s'}
+(\d_0^\v)^{|\mathbf{s}|-p}U_{(p)}
+V_\mathbf{s}.$$
From Proposition~\ref{P8.10}, a morphism~$\theta$ of abelian groups with domain $\Delta(|\mathbf{s}|)$ such that $\theta|U_\mathbf{s}=0$ is equivalent to a pair of morphisms $\theta'$~and~$\theta''$ with domains $\Delta(|\mathbf{s'}|)$ and $\Delta(p)$ such that
$$\theta'|U_\mathbf{s'}=0,\quad 
\theta''|U_{(p)}=0,\quad 
\theta'(\d_0^\v)^{|\mathbf{s}|-p}=\theta''(\d_1^\v)^{p-q}.$$
From this it follows that a morphism~$\chi$ of abelian groups with domain~$S_\mathbf{s}$ is equivalent to a pair of morphisms $\chi'$~and~$\chi''$ with domains $S_\mathbf{s'}$~and~$S_{(p)}$ which agree on~$S_{(q)}$. Therefore $R_\mathbf{s}$~is a push-out as a square of abelian groups. 

It also follows from Proposition~\ref{P8.10} that the prescribed submonoid in~$S_\mathbf{s}$ is generated by the images of the prescribed submonoids in $S_\mathbf{s'}$~and~$S_{(p)}$, because the basis elements in $\Delta(|\mathbf{s}|)$ are congruent modulo~$V_\mathbf{s}$ to sums of basis elements in
$$(\d_{|\mathbf{s}|-p+1}^\v)^{p-q}\Delta(|\mathbf{s'}|)
+(\d_0^\v)^{|\mathbf{s}|-p}\Delta(p).$$

This completes the proof.
\end{proof}

We can now give the main result in this section.

\begin{theorem} \label{T8.12}
Let $\mathbf{s}$ be an up-down vector with last term~$p$ and let
$$a_m
=[\,|\mathbf{s}|-m,\,|\mathbf{s}|-m+1,\,|\mathbf{s}|-m+2,\,\ldots,\,|\mathbf{s}|\,]
\quad 
(0\leq m\leq p).$$
Then $S_\mathbf{s}$~is an $\mathbf{s}$-simple chain complex with final generator~$a$ such that
$$a_m+U_\mathbf{s}=(\d^+)^{p-m}a\quad (0\leq m\leq p).$$

If $\mathbf{s}$~has more than one term, then $R_\mathbf{s}$ is an $\mathbf{s}$-simple square of chain complexes.
\end{theorem}

\begin{proof}
The proof is by induction on the number of terms in~$\mathbf{s}$. Proposition~\ref{P8.7} gives the result for the case $\mathbf{s}=(p)$. From now on let $\mathbf{s}=(\mathbf{s'},q,p)$, let $p'$ be the last term of~$\mathbf{s'}$, and recall that the square~$R_\mathbf{s}$ has the form
$$\xymatrix{
S_{(q)} \ar_{\pi}[d] \ar^{\rho}[r] & S_{(p)} \ar^{\tau}[d] \\
S_\mathbf{s'} \ar_{\sigma}[r] & S_\mathbf{s}.
}$$
From the inductive hypothesis and the one-term case we see that $S_\mathbf{s'}$, $S_{(q)}$ and~$S_{(p)}$ are $\mathbf{s'}$-simple, $(q)$-simple and $(p)$-simple; let the final generators be $a'$, $a^0$ and~$a''$. 
For $0\leq m\leq q$ it follows from the equalities
\begin{align*}
&(\d_0^\v)^{|\mathbf{s}|-p}[q-m,q-m+1,\ldots,q]
=[\,|\mathbf{s'}|-m,|\mathbf{s'}|-m+1,\ldots,|\mathbf{s'}|\,],\\
&(\d_1^\v)^{p-q}[0,q-m+1,\ldots,q]
=[0,p-m+1,\ldots,p]
\end{align*}
that
\begin{align*}
&\pi(\d^+)^{q-m}a^0=(\d^+)^{p'-m}a',\\
&\pi(\d^-)^{q-m}a^0=(\d^-)^{p-m}a''.
\end{align*}
It now follows from Proposition~\ref{P8.11} that $R_\mathbf{s}$~is an $\mathbf{s}$-simple square (see Definition~\ref{D5.5}). By Proposition~\ref{P5.6}, this makes $S_\mathbf{s}$ an $\mathbf{s}$-simple chain complex with final generator~$a$ such that
$$\tau(\d^+)^{p-m}a''=(\d^+)^{p-m}a\quad (0\leq m\leq p).$$

For $0\leq m\leq p$ let
$$a_m''=[p-m,p-m+1,\ldots,p]\in\Delta(p),$$
so that $a_m''+U_{(p)}=(\d^+)^{p-m}a''$ by Proposition~\ref{P8.7}; then
$$a_m+U_\mathbf{s}
=(\d_0^\v)^{|\mathbf{s}|-p}a_m''+U_\mathbf{s}
=\tau(a_m''+U_{(p)})
=\tau(\d^+)^{p-m}a''
=(\d^+)^{p-m}a.$$

This completes the proof.
\end{proof}

\section{Combined operations in sets with complicial identities} \label{S9}

In Theorem~\ref{T8.12} we have constructed an $\mathbf{s}$-simple chain complex~$S_\mathbf{s}$ as a quotient of a simplex,
$$S_\mathbf{s}=\Delta(|\mathbf{s}|)/U_\mathbf{s}.$$
We really want to express~$S_\mathbf{s}$ as a retract of $\Delta(|\mathbf{s}|)$; that is, we want to construct an idempotent endomomorphism of $\Delta(|\mathbf{s}|)$ with kernel~$U_\mathbf{s}$. In this section we construct the corresponding operation on $|\mathbf{s}|$-dimen\-sio\-nal elements in sets with complicial identities; this operation will be denoted~$\Psi_\mathbf{s}$. We use the axioms of Definition~\ref{D3.1} throughout.

We will construct~$\Psi_\mathbf{s}$ by iterating wedge operations. There are two basic families of iterated wedges, and we will now describe the first of these families.

\begin{notation} \label{N9.1}
For $0<i\leq i+j\leq m$, let $\tilde\phi_{i,j}$~and~$\phi_{i,j}$ be the operations on $m$-dimensional elements in sets with complicial identities given by
\begin{align*}
&\tilde\phi_{i,0}x=\e_{i-1}x,\\
&\tilde\phi_{i,j}x=\tilde\phi_{i,j-1}\d_{i+1}x\w_i x\quad (j>0),\\
&\phi_{i,j}x=\d_{i+1}\tilde\phi_{i,j}x.
\end{align*}
\end{notation}

The wedge in the formula for $\tilde\phi_{i,j}x$ exists by an induction on~$j$: if $\tilde\phi_{i,j-1}x$ exists and is given by the stated formula, then $$\d_i\tilde\phi_{i,j-1}\d_{i+1}x=\d_{i+1}x,$$
and the wedge $\tilde\phi_{i,j-1}\d_{i+1}x\w_i x$ therefore exists.

We will now compute some faces and some fixed point sets. 

\begin{proposition} \label{P9.2}
The operations $\phi_{i,j}$~and~$\tilde\phi_{i,j}$ are such that
$$\d_i\phi_{i,j}=\d_i,\quad \d_{i+1}^j\phi_{i,j}=\d_{i+1}^{j+1}\tilde\phi_{i,j}=\e_{i-1}\d_i^{j+1}.$$
\end{proposition}

\begin{proof}
The first formula holds because
$$
\d_i\phi_{i,j}x
=\d_i\d_{i+1}\tilde\phi_{i,j}x=\d_i\d_i\tilde\phi_{i,j}x=\d_i x.$$
The second formula holds by induction on~$j$: we certainly have
$$\phi_{i,0}x=\d_{i+1}\tilde\phi_{i,0}x=\d_{i+1}\e_{i-1}x=\e_{i-1}\d_i x,$$
and for $j>0$ we have
\begin{align*}
\d_{i+1}^j\phi_{i,j}x
&=\d_{i+1}^{j+1}\tilde\phi_{i,j}x\\
&=\d_{i+1}^j\d_{i+2}(\tilde\phi_{i,j-1}\d_{i+1}x\w_i x)\\
&=\d_{i+1}^j\tilde\phi_{i,j-1}\d_{i+1}x\\
&=\e_{i-1}\d_i^j\d_{i+1}x\\
&=\e_{i-1}\d_i^{j+1}x.
\end{align*}
\end{proof}

\begin{proposition} \label{P9.3}
Let $x$ be a member of a set with complicial identities. Then
$$\tilde\phi_{i,j}x=\e_i x\iff\phi_{i,j}x=x\iff \d_{i+1}^j x\in\im\e_{i-1}.$$
\end{proposition}

\begin{proof}
Suppose that $\tilde\phi_{i,j}x=\e_i x$. Then 
$$\phi_{i,j}x=\d_{i+1}\e_i x=x.$$

Suppose that $\phi_{i,j}x=x$. Then
$$\d_{i+1}^j x=\d_{i+1}^j\phi_{i,j}x=\e_{i-1}\d_i^{j+1}x\in\im\e_{i-1}.$$

It now suffices to show that 
$$\d_{i+1}^j x\in\im\e_{i-1}\Rightarrow\tilde\phi_{i,j}x=\e_i x.$$
We argue by induction on~$j$. 

Suppose that $x\in\im\e_{i-1}$. Since $\e_{i-1}\e_{i-1}=\e_i\e_{i-1}$, it follows that $\tilde\phi_{i,0}x=\e_i x$. 

Now suppose that $\d_{i+1}^j x\in\im\e_{i-1}$ for some $j>0$. It follows from the inductive hypothesis that $\tilde\phi_{i,j-1}\d_{i+1}x=\e_i\d_{i+1}x$, and it then follows that
$$\tilde\phi_{i,j}x=\e_i\d_{i+1}x\w_i x=\e_i x.$$

This completes the proof.
\end{proof}

We now consider the second basic family of iterated wedges.

\begin{notation} \label{N9.4}
For integers $k,l\geq 0$ let $\w_{k,l}$ be the partial binary operation in sets with complicial identities such that $x\w_{k,l}y$ is defined when
$$\d_0^k x=\d_1^l y$$
and such that
\begin{align*}
&x\w_{k,l}y=x\quad (l=0),\\
&x\w_{k,l}y=y\quad (k=0),\\
&x\w_{k,l}y=(x\w_{k,l-1}\d_1 y)\w_{k-1}(\d_{k-1}x\w_{k-1,l}y)\quad (k,l>0),\\
&\d_{k-1}(x\w_{k,l}y)=\d_{k-1}x\w_{k-1,l}y\quad (k>0),\\
&\d_{k+1}(x\w_{k,l}y)=x\w_{k,l-1}\d_1 y\quad (l>0).
\end{align*}
\end{notation}

To justify this definition, let $x$~and~$y$ be such that $\d_0^k x=\d_1^l y$; we must show that the stated conditions make sense and are consistent. We do this by induction on $k$~and~$l$.

Suppose that $k=l=0$. Then $x=y$, so the conditions $x\w_{k,l}y=x$ and $x\w_{k,l}y=y$ are consistent.

Suppose that $k>0$ and $l=0$. Then $\d_0^{k-1}\d_{k-1}x=\d_0^k x=\d_1^l y$; hence, by the inductive hypothesis, $\d_{k-1}x\w_{k-1,l}y$ exists and is equal to $\d_{k-1}x$. The conditions $x\w_{k,l}y=x$ and $\d_{k-1}(x\w_{k,l}y)=\d_{k-1}x\w_{k-1,l}y$ therefore make sense and are consistent.

Suppose that $k=0$ and $l>0$. Then $\d_0^k x=\d_1^l y=\d_1^{l-1}\d_1 y$; hence $x\w_{k,l-1}\d_1 y$ exists and is equal to $\d_1 y$. The conditions $x\w_{k,l}y=y$ and $\d_{k+1}(x\w_{k,l}y)=\d_1 y$ therefore make sense and are consistent.

Finally suppose that $k>0$ and $l>0$. As in the previous cases, the expressions $\d_{k-1}x\w_{k-1,l}y$ and $x\w_{k,l-1}\d_1 y$ make sense. We also have
$$\d_{k-1}(x\w_{k,l-1}\d_1 y)=\d_{k-1}x\w_{k-1,l-1}\d_1 y=\d_k(\d_{k-1}x\w_{k-1,l}y),$$
so the conditions
\begin{align*}
&x\w_{k,l}y=(x\w_{k,l-1}\d_1 y)\w_{k-1}(\d_{k-1}x\w_{k-1,l}y),\\
&\d_{k-1}(x\w_{k,l}y)=\d_{k-1}x\w_{k-1,l}y,\\
&\d_{k+1}(x\w_{k,l}y)=x\w_{k,l-1}\d_1 y
\end{align*}
all make sense. It is also clear that they are consistent.

Each of these binary operations determines its own factors.

\begin{proposition} \label{P9.5}
If $x\w_{k,l}y$ is defined, then
$$\d_{k+1}^l(x\w_{k,l}y)=x,\quad \d_0^k(x\w_{k,l}y)=y.$$
\end{proposition}

\begin{proof}
The first equality is proved by induction on~$l$; for $l=0$ it is obvious, and for $l>0$ we have
$$\d_{k+1}^l(x\w_{k,l}y)
=\d_{k+1}^{l-1}\d_{k+1}(x\w_{k,l}y)
=\d_{k+1}^{l-1}(x\w_{k,l-1}\d_1 y)
=x.$$
The second equality is similarly proved by induction on~$k$; for $k=0$ it is obvious, and for $k>0$ we have
$$\d_0^k(x\w_{k,l}y)
=\d_0^{k-1}\d_{k-1}(x\w_{k,l}y)
=\d_0^{k-1}(\d_{k-1}x\w_{k,l}y)
=y.$$
\end{proof}

Because of the simplicial identities $\d_0^k\d_{k+1}^l=\d_1^l\d_0^k$ there are everywhere defined unary operations as follows.

\begin{notation} \label{N9.6}
If $k$~and~$l$ are nonnegative integers and if $x$~is an element of dimension at least $k+l$ in a set with complicial identities, then
$$w_{k,l}x=\d_{k+1}^l x\w_{k,l}\d_0^k x.$$
\end{notation}

\begin{proposition} \label{P9.7}
The operations~$w_{k,l}$ are idempotent operations such that
$$\d_{k+1}^l w_{k,l}=\d_{k+1}^l,\quad \d_0^k w_{k,l}=\d_0^k.$$
If $X$~is a set with complicial identities and if $q\geq 0$ then the square
$$\xymatrix{
w_{k,l}X_{k+l+q} \ar_{\d_{k+1}^l}[d] \ar^{\d_0^k}[r]
& X_{l+q} \ar^{\d_1^l}[d] \\
X_{k+q} \ar_{\d_0^k}[r] 
& X_q
}$$
is a pull-back square.
\end{proposition}

\begin{proof}
For all~$x$ it follows from Proposition~\ref{P9.5} that
\begin{align*}
&\d_{k+1}^l w_{k,l}x=\d_{k+1}^l(\d_{k+1}^l x\w_{k,l}\d_0^k x)=\d_{k+1}^l x,\\
&\d_0^k w_{k,l}x=\d_0^k(\d_{k+1}^l x\w_{k,l}\d_0^k x)=\d_0^k x;
\end{align*}
therefore $\d_{k+1}^l w_{k,l}=\d_{k+1}^l$ and $\d_0^k w_{k,l}=\d_0^k$. It then follows that 
$$w_{k,l}w_{k,l}x
=\d_{k+1}^l w_{k,l}x\w_{k,l}\d_0^k w_{k,l}x
=\d_{k+1}^l x\w_{k,l}\d_0^k x
=w_{k,l}x;$$
therefore $w_{k,l}$~is idempotent. If $x\in w_{k,l}X_{k+l+q}$ then certainly $\d_0^k\d_{k+1}^l x=\d_1^l\d_0^k x$. Conversely, if $y\in X_{k+q}$ and $z\in X_{l+q}$ are such that $\d_0^k y=\d_1^l z$ then 
$$x=w_{k,l}(y\w_{k+l} z)$$
is a member of $w_{k,l}X_{k+l+q}$ such that $\d_{k+1}^l x=y$ and $\d_0^k x=z$ and it is clearly the unique such member; therefore the square is a pull-back square.
\end{proof}

We will now combine the operations $\phi_{i,j}$~and~$w_{k,l}$. Let $\mathbf{s}$ be an up-down vector. We will construct operators~$\psi_{i,\mathbf{s}}$ on $|\mathbf{s}|$-dimensional elements for $0<i<|\mathbf{s}|$. We begin with large values of~$i$ and work downwards.

\begin{notation} \label{N9.8}
Let $\mathbf{s}$ be an up-down vector with last term~$p$ and let $x$ be an $|\mathbf{s}|$-dimensional element in a set with complicial identities. Then
$$\psi_{i,\mathbf{s}}x=\phi_{i,1}x\quad (|\mathbf{s}|-p<i<|\mathbf{s}|).$$
\end{notation}

\begin{notation} \label{N9.9}
Let $\mathbf{s}=(\mathbf{s'},q,p)$ and let $x$ be an $|\mathbf{s}|$-dimensional element in a set with complicial identities. Then
$$\psi_{|\mathbf{s}|-p,\mathbf{s}}x
=\begin{cases}
w_{|\mathbf{s}|-p,p-q}x& (q=0),\\
w_{|\mathbf{s}|-p,p-q}\phi_{|\mathbf{s}|-p,p-q+1}x& (q>0).
\end{cases}$$
\end{notation}

\begin{proposition} \label{P9.10}
If $\mathbf{s}$~is an up-down vector with last term~$p$ and if $\psi_{i,\mathbf{s}}$~is an operator with $|\mathbf{s}|-p\leq i<|\mathbf{s}|$, then
$$\d_0^j\psi_{i,\mathbf{s}}=\d_0^j\quad (i<j\leq |\mathbf{s}|).$$
\end{proposition}

\begin{proof}
In cases with $i>|\mathbf{s}|-p$ the result holds by Proposition~\ref{P9.2} because 
$$\d_0^j\psi_{i,\mathbf{s}}=\d_0^{j-1}\d_i\phi_{i,1}=\d_0^{j-1}\d_i=\d_0^j.$$

In cases with $i=|\mathbf{s}|-p$ we have 
$$\psi_{i,\mathbf{s}}=w_{|\mathbf{s}|-p,p-q}$$ or we have 
$$\psi_{i,\mathbf{s}}=w_{|\mathbf{s}|-p,p-q}\phi_{|\mathbf{s}|-p,p-q+1}$$ 
for some~$q$. The result now holds because
$$\d_0^j w_{|\mathbf{s}|-p,p-q}
=\d_0^{j-|\mathbf{s}|+p}\d_0^{|\mathbf{s}|-p}w_{|\mathbf{s}|-p,p-q}
=\d_0^{j-|\mathbf{s}|+p}\d_0^{|\mathbf{s}|-p}
=\d_0^j$$
by Proposition~\ref{P9.7} and because $\d_0^j\phi_{|\mathbf{s}|-p,p-q+1}=\d_0^j$ by Proposition~\ref{P9.2} as before.
\end{proof}

In the remaining cases we use induction on the number of terms in~$\mathbf{s}$.

\begin{notation} \label{N9.11}
Let $\mathbf{s}=(\mathbf{s'},q,p)$, let $x$ be an $|\mathbf{s}|$-dimensional element in a set with complicial identities, and let $i$ be an integer with $0<i<|\mathbf{s}|-p$. Then $\psi_{i,\mathbf{s}}x$ is the element such that
$$\psi_{i,\mathbf{s}}x
=\psi_{i,\mathbf{s'}}\d_{|\mathbf{s}|-p+1}^{p-q}x
\w_{|\mathbf{s}|-p,p-q}\d_0^{|\mathbf{s}|-p}x$$
and
$$\d_0^j\psi_{i,\mathbf{s}}x=\d_0^j x\quad (|\mathbf{s}|-p\leq j\leq |\mathbf{s}|).$$
\end{notation}

To justify this definition, we must show that the two conditions make sense and are consistent. We argue by induction on the number of terms in~$\mathbf{s}$. Let $p'$ be the last term in~$\mathbf{s'}$, so that
$$|\mathbf{s}|-p=|\mathbf{s'}|-q>|\mathbf{s'}|-p'.$$
If $i\geq |\mathbf{s'}|-p'$ then $\d_0^{|\mathbf{s}|-p}\psi_{i,\mathbf{s'}}=\d_0^{|\mathbf{s}|-p}$ by Proposition~\ref{P9.10}; if $i<|\mathbf{s}|-p'$ then $\d_0^{|\mathbf{s}|-p}\psi_{i,\mathbf{s'}}=\d_0^{|\mathbf{s}|-p}$ by the inductive hypothesis. In all cases it follows that
$$\d_0^{|\mathbf{s}|-p}\psi_{i,\mathbf{s'}}\d_{|\mathbf{s}|-p+1}^{p-q}x
=\d_0^{|\mathbf{s}|-p}\d_{|\mathbf{s}|-p+1}^{p-q}x
=\d_1^{p-q}\d_0^{|\mathbf{s}|-p}x.$$
The condition
$$\psi_{i,\mathbf{s}}x
=\psi_{i,\mathbf{s'}}\d_{|\mathbf{s}|-p+1}^{p-q}x
\w_{|\mathbf{s}|-p,p-q}\d_0^{|\mathbf{s}|-p}x$$
therefore makes sense because the iterated wedge exists. This condition actually implies the other condition by Proposition~\ref{P9.7}, because for $|\mathbf{s}|-p\leq j\leq |\mathbf{s}|$ we have
\begin{align*}
&\d_0^j
(\psi_{i,\mathbf{s'}}\d_{|\mathbf{s}|-p+1}^{p-q}x
\w_{|\mathbf{s}|-p,p-q}\d_0^{|\mathbf{s}|-p}x)\\
&=\d_0^{j-|\mathbf{s}|+p}\d_0^{|\mathbf{s}|-p}
(\psi_{i,\mathbf{s'}}\d_{|\mathbf{s}|-p+1}^{p-q}x
\w_{|\mathbf{s}|-p,p-q}\d_0^{|\mathbf{s}|-p}x)\\
&=\d_0^{j-|\mathbf{s}|+p}\d_0^{|\mathbf{s}|-p}x\\
&=\d_0^j x.
\end{align*}

Finally we construct the operation~$\Psi_\mathbf{s}$ as an iterated composite.

\begin{notation} \label{N9.12}
Let $\mathbf{s}$ be an up-down vector. Then $\Psi_\mathbf{s}$~is the operation on $|\mathbf{s}|$-dimensional elements in sets with complicial identities given by
$$\Psi_\mathbf{s}
=(\psi_{1,\mathbf{s}})
(\psi_{2,\mathbf{s}}\psi_{1,\mathbf{s}})
(\psi_{3,\mathbf{s}}\psi_{2,\mathbf{s}}\psi_{1,\mathbf{s}})
\ldots
(\psi_{|\mathbf{s}|-1,\mathbf{s}}\ldots\psi_{2,\mathbf{s}}\psi_{1,\mathbf{s}})$$
(to be interpreted as the identity when $|\mathbf{s}|\leq 1$).
\end{notation}

\section{The induced morphisms between simplexes} \label{S10}

In Section~\ref{S9} we have constructed operations in sets with complicial identities. We will now give some results concerning the induced morphisms between simplexes.

\begin{proposition} \label{P10.1}
If $\theta_1$~and~$\theta_2$ are operations in sets with complicial identities such that $\d_i\theta_1=\d_{i+1}\theta_2$ and if $\theta$~is the operation given by
$$\theta x=\theta_1 x\w_i\theta_2 x,$$
then
$$\theta^\v
=\theta_1^\v\e_{i+1}^\v-\theta_1^\v\d_i^\v(\e_i^\v)^2
+\theta_2^\v\e_i^\v
=\theta_1^\v\e_{i+1}^\v-\theta_2^\v\d_{i+1}^\v(\e_i^\v)^2
+\theta_2^\v\e_i^\v.$$
\end{proposition}

\begin{proof}
This follows from the construction of wedges in sets with complicial identities of the form $\lambda K$; see Notation~\ref{N6.9}.
\end{proof}

\begin{proposition} \label{P10.2}
If $j>0$ then
$$\phi_{i,j}^\v=(\d_{i+1}^\v)^j[\e_{i-1}^\v-\e_i^\v](\e_{i+1}^\v)^{j-1}+\id.$$
\end{proposition}

\begin{proof}
Recall from Notation~\ref{N9.1} that $\phi_{i,j}=\d_{i+1}\tilde\phi_{i,j}$ with
\begin{align*}
&\tilde\phi_{i,0}x=\e_{i-1}x,\\
&\tilde\phi_{i,j}x=\tilde\phi_{i,j-1}\d_{i+1}x\w_i x\quad (j>0).
\end{align*}
It follows from Proposition~\ref{P10.1} by induction on~$j$ that
$$\tilde\phi_{i,j}^\v
=(\d_{i+1}^\v)^j\bigl[\e_{i-1}^\v-\e_i^\v)\bigr](\e_{i+1}^\v)^j+\e_i^\v
\quad (j\geq 0),$$
and for $j>0$ it then follows that
$$\phi_{i,j}^\v
=\tilde\phi_{i,j}^\v\d_{i+1}^\v
=(\d_{i+1}^\v)^j\bigl[\e_{i-1}^\v-\e_i^\v)\bigr](\e_{i+1}^\v)^{j-1}+\id.$$
\end{proof}

\begin{proposition} \label{P10.3}
If $\theta_1$~and~$\theta_2$ are operations in sets with complicial identities such that $\d_0^k\theta_1=\d_1^l\theta_2$ and if $\theta$~is the operation given by
$$\theta x=\theta_1 x\w_{k,l}\theta_2 x,$$
then
\begin{align*}
\theta^\v
&=\theta_1^\v(\e_k^\v)^l-\theta_1^\v(\d_0^\v)^k(\e_0^\v)^{k+l}
+\theta_2^\v(\e_0^\v)^k\\
&=\theta_1^\v(\e_k^\v)^l-\theta_2^\v(\d_1^\v)^k(\e_0^\v)^{k+l}
+\theta_2^\v(\e_0^\v)^k.
\end{align*}
\end{proposition}

\begin{proof}
This follows by induction from Proposition~\ref{P10.1} using the formulae
\begin{align*}
&\theta_1 x\w_{k,0}\theta_2 x=\theta_1 x,\\ 
&\theta_1 x\w_{0,l}\theta_2 x=\theta_2 x,\\
&\theta_1 x\w_{k,l}\theta_2 x
=(\theta_1 x\w_{k,l-1}\d_1\theta_2 x)\w_{k-1}(\d_{k-1}\theta_1 x\w_{k-1,l}\theta_2 x)
\quad (k,l>0)
\end{align*}
(see Notation~\ref{N9.4}).
\end{proof}

\begin{proposition} \label{P10.4}
The morphisms~$w_{k,l}^\v$ are given by
$$w_{k,l}^\v
=(\d_{k+1}^\v)^l(\e_k^\v)^l
-(\d_{k+1}^\v)^l(\d_0^\v)^k(\e_0^\v)^{k+l}
+(\d_0^\v)^l(\e_0^\v)^l.$$
\end{proposition}

\begin{proof}
This follows from Proposition~\ref{P10.3}, because
$$w_{k,l}x=\d_{k+1}^l x\w_{k,l}\d_0^k x$$
(see Notation~\ref{N9.6}).
\end{proof}

\begin{proposition} \label{P10.5}
If $\mathbf{s}=(\mathbf{s'},q,p)$ and if $0<i<|\mathbf{s}|-p$, then
\begin{align*}
&\psi_{i,\mathbf{s}}^\v -w_{|\mathbf{s}|-p,p-q}^\v
=(\d_{|\mathbf{s}|-p+1}^\v)^{p-q}
[\psi_{i,\mathbf{s'}}^\v-\id]
(\e_{|\mathbf{s}|-p}^\v)^{p-q},\\
&\psi_{i,\mathbf{s}}^\v(\d_{|\mathbf{s}|-p+1}^\v)^{p-q}
=(\d_{|\mathbf{s}|-p+1}^\v)^{p-q}\psi_{i,\mathbf{s'}}^\v,\\
&\psi_{i,\mathbf{s}}^\v(\d_0^\v)^{|\mathbf{s}|-p}=(\d_0^\v)^{|\mathbf{s}|-p}.
\end{align*}
\end{proposition}

\begin{proof}
By Definition (see Notation~\ref{N9.11}), 
$$\psi_{i,\mathbf{s}}x
=\psi_{i,\mathbf{s'}}\d_{|\mathbf{s}|-p+1}^{p-q}x
\w_{|\mathbf{s}|-p,p-q}\d_0^{|\mathbf{s}|-p}x.$$
The formula for $\psi_{i,\mathbf{s}}^\v -w_{|\mathbf{s}|-p,p-q}^\v$ follows from Propositions \ref{P10.3} and~\ref{P10.4}. The other two formulae hold because $\d_{|\mathbf{s}|-p+1}^{p-q}\psi_{i,\mathbf{s}}=\psi_{i,\mathbf{s'}}\d_{|\mathbf{s}|-p+1}^{p-q}$ and $\d_0^{|\mathbf{s}|-p}\psi_{i,\mathbf{s}}=\d_0^{|\mathbf{s}|-p}$ (see Proposition~\ref{P9.5}).
\end{proof}

\begin{proposition} \label{P10.6}
If $\mathbf{s}=(\mathbf{s'},q,p)$ with $q=0$, then
\begin{align*}
&\psi_{|\mathbf{s}|-p,\mathbf{s}}^\v-w_{|\mathbf{s}|-p,p-q}^\v=0,\\
&\psi_{|\mathbf{s}|-p,\mathbf{s}}^\v(\d_{|\mathbf{s}|-p+1}^\v)^{p-q}
=(\d_{|\mathbf{s}|-p+1}^\v)^{p-q},\\
&\psi_{|\mathbf{s}|-p,\mathbf{s}}^\v(\d_0^\v)^{|\mathbf{s}|-p}
=(\d_0^\v)^{|\mathbf{s}|-p}.
\end{align*}
\end{proposition}

\begin{proof}
By Notation~\ref{N9.9}, $\psi_{|\mathbf{s}|-p,\mathbf{s}}=w_{|\mathbf{s}|-p,p-q}$. The last two formulae follow from Proposition~\ref{P10.4}.
\end{proof}

\begin{proposition} \label{P10.7}
If $\mathbf{s}=(\mathbf{s'},q,p)$ with $q>0$, then
\begin{align*}
&\psi_{|\mathbf{s}|-p,\mathbf{s}}^\v-w_{|\mathbf{s}|-p,p-q}^\v
=(\d_{|\mathbf{s}|-p+1}^\v)^{p-q}[\psi_{|\mathbf{s}|-p,\mathbf{s'}}^\v-\id]
(\e_{|\mathbf{s}|-p+1}^\v)^{p-q}w_{|\mathbf{s}|-p,p-q}^\v,\\
&\psi_{|\mathbf{s}|-p,\mathbf{s}}^\v(\d_{|\mathbf{s}|-p+1}^\v)^{p-q}
=(\d_{|\mathbf{s}|-p+1}^\v)^{p-q}\psi_{|\mathbf{s}|-p,\mathbf{s'}}^\v,\\
&\bigl[\psi_{|\mathbf{s}|-p,\mathbf{s}}^\v-\id\bigr](\d_0^\v)^{|\mathbf{s}|-p}\\
&\quad{}=(\d_{|\mathbf{s}|-p+1}^\v)^{p-q+1}
\bigl[\e_{|\mathbf{s}|-p-1}^\v-\e_{|\mathbf{s}|-p}^\v]
(\e_{|\mathbf{s}|-p+1}^\v)^{p-q}(\d_0^\v)^{|\mathbf{s}|-p}.
\end{align*}
\end{proposition}

\begin{proof}
In this case $\psi_{|\mathbf{s}|-p,\mathbf{s}}=w_{|\mathbf{s}|-p,p-q+1}\phi_{|\mathbf{s}|-p,p-q+1}$ (see Notation~\ref{N9.9}), hence
\begin{align*}
&\psi_{|\mathbf{s}|-p,\mathbf{s}}^\v-w_{|\mathbf{s}|-p,p-q}^\v\\
&\qquad{}=[\phi_{|\mathbf{s}|-p,p-q+1}^\v-\id]w_{|\mathbf{s}|-p,p-q}^\v\\
&\qquad{}=(\d_{|\mathbf{s}|-p+1}^\v)^{p-q+1}[\e_{|\mathbf{s}|-p-1}^\v-\e_{|\mathbf{s}|-p}^\v](\e_{|\mathbf{s}|-p+1}^\v)^{p-q}w_{|\mathbf{s}|-p,p-q}^\v\\
&\qquad{}=(\d_{|\mathbf{s}|-p+1}^\v)^{p-q}[\phi_{|\mathbf{s}|-p,1}^\v-\id](\e_{|\mathbf{s}|-p+1}^\v)^{p-q}w_{|\mathbf{s}|-p,p-q}^\v
\end{align*}
by Proposition~\ref{P10.2}. We also have $\phi_{|\mathbf{s}|-p,1}=\psi_{|\mathbf{s}|-p,\mathbf{s'}}$ by Notation~\ref{N9.8} because $|\mathbf{s}|-p>|\mathbf{s'}|-p'$, where $p'$~is the last term in~$\mathbf{s'}$. The results follow.
\end{proof}

\begin{proposition} \label{P10.8}
If $\mathbf{s}=(\mathbf{s'},q,p)$ then
\begin{align*}
&\psi_{i,\mathbf{s}}^\v(\d_{|\mathbf{s}|-p+1}^\v)^{p-q}
=(\d_{|\mathbf{s}|-p+1}^\v)^{p-q}\quad (|\mathbf{s}|-p<i<|\mathbf{s}|-q),\\
&\psi_{|\mathbf{s}|-q,\mathbf{s}}^\v(\d_{|\mathbf{s}|-p+1}^\v)^{p-q}
=(\d_{|\mathbf{s}|-p+1}^\v)^{p-q}\quad (q>0),\\
&\psi_{i,\mathbf{s}}^\v(\d_{|\mathbf{s}|-p+1}^\v)^{p-q}
=(\d_{|\mathbf{s}|-p+1}^\v)^{p-q}\psi_{i-p+q,\mathbf{s'}}^\v\quad 
(|\mathbf{s}|-q+1<i<|\mathbf{s}|),\\
&w_{|\mathbf{s}|-p,p-q}^\v
\psi_{|\mathbf{s}|-q+1,\mathbf{s}}^\v
(\d_{|\mathbf{s}|-p+1}^\v)^{p-q}
=(\d_{|\mathbf{s}|-p+1}^\v)^{p-q}
\psi_{|\mathbf{s}|-p+1,\mathbf{s'}}^\v\\
&\qquad{}+(\d_0^\v)^{|\mathbf{s}|-p}
\bigl[\phi_{p-q+1,1}^\v(\d_1^\v)^{p-q}-(\d_1^\v)^{p-q}\phi_{1,1}^\v\bigr]
(\e_0^\v)^{|\mathbf{s}|-p}\quad (q>1),\\
&\psi_{|\mathbf{s}|-q,\mathbf{s}}^\v(\d_0^\v)^{|\mathbf{s}|-p}
=(\d_0^\v)^{|\mathbf{s}|-p}\psi_{i-|\mathbf{s}|+p,(p)}^\v\quad 
(|\mathbf{s}|-p<i<|\mathbf{s}|).
\end{align*}
\end{proposition}

\begin{proof}
According to Notation~\ref{N9.8} we have
$$\psi_{i,\mathbf{s}}=\phi_{i,1},\quad 
\psi_{i-|\mathbf{s}|+p,(p)}=\phi_{i-|\mathbf{s}|+p,1}$$
for $|\mathbf{s}|-p<i<|\mathbf{s}|$. We also have $\psi_{i-p+q,\mathbf{s'}}=\phi_{i-p+q,1}$ for $|\mathbf{s}|-q<i<|\mathbf{s}|$ because
$$i-p+q>|\mathbf{s}|-p>|\mathbf{s'}|-p',$$
where $p'$~is the last term in~$\mathbf{s'}$. The results now follow from Propositions \ref{P10.2} and~\ref{P10.4}.
\end{proof}

Finally in this section, we consider the action of~$\psi_{i,\mathbf{s}}^\v$ on~$V_\mathbf{s}$ in the case $\mathbf{s}=(\mathbf{s'},q,p)$; we recall from Notation~\ref{N8.4} that $V_\mathbf{s}$~is the subcomplex of $\Delta(|\mathbf{s}|)$ generated by the basis elements
$$[i_0,\ldots,i_{r-1},|\mathbf{s}|-p,i_{r+1},\ldots,i_m]$$
with $0\leq i_{r-1}<|\mathbf{s}|-p<i_{r+1}\leq |\mathbf{s}|-q$.

\begin{proposition} \label{P10.9}
If $\mathbf{s}=(\mathbf{s'},q,p)$ then
\begin{align*}
&\psi_{i,\mathbf{s}}^\v V_\mathbf{s}=0\quad (0<i\leq |\mathbf{s}|-p),\\
&\psi_{i,\mathbf{s}}^\v V_\mathbf{s}\subset V_\mathbf{s}\quad 
(|\mathbf{s}|-p<i<|\mathbf{s}).
\end{align*}
\end{proposition}

\begin{proof}
For $0<i<|\mathbf{s}|-p$ we have
\begin{multline*}
\psi_{i,\mathbf{s}}^\v
=(\d_{|\mathbf{s}|-p+1}^\v)^{p-q}\psi_{i,\mathbf{s'}}^\v(\e_{|\mathbf{s}|-p}^\v)^{p-q}\\
{}-(\d_{|\mathbf{s}|-p+1}^\v)^{p-q}(\d_0^\v)^{|\mathbf{s}|-p}(\e_0^\v)^{|\mathbf{s}|-q}
+(\d_0^\v)^{|\mathbf{s}|-p}(\e_0^\v)^{|\mathbf{s}|-p}
\end{multline*}
(see Notation~\ref{N9.11} and Proposition~\ref{P9.3}), hence $\psi_{i,\mathbf{s}}^\v V_\mathbf{s}=0$. In the same way $w_{|\mathbf{s}|-p,p-q}^\v V_\mathbf{s}=0$ by Proposition~\ref{P10.4}; hence, by Notation~\ref{N9.9}, $\psi_{|\mathbf{s}|-p,\mathbf{s}}^\v V_\mathbf{s}=0$. For $|\mathbf{s}|-p<i<|\mathbf{s}|$ we have $\psi_{i,\mathbf{s}}=\phi_{i,1}$ by Notation~\ref{N9.8}, hence $\psi_{i,\mathbf{s}}^\v V_\mathbf{s}\subset V_\mathbf{s}$ by Proposition~\ref{P10.2}.
\end{proof}

\section{Simple chain complexes as retracts of simplexes} \label{S11}

Given an up-down vector~$\mathbf{s}$, we have shown in Section~\ref{S8} that the simplex $\Delta(|\mathbf{s}|)$ has an $\mathbf{s}$-simple quotient 
$$S_\mathbf{s}=\Delta(|\mathbf{s}|)/U_\mathbf{s}.$$
In Section~\ref{S9} we have constructed an operation~$\Psi_\mathbf{s}$ on $|\mathbf{s}|$-dimensional elements in sets with complicial identities. By Theorem~\ref{T6.11} there is a corresponding endomorphism~$\Psi_\mathbf{s}^\v$ of $\Delta(|\mathbf{s}|)$. We will now show that $S_\mathbf{s}$~is a retract of $\Delta(|\mathbf{s}|)$ by showing that $\Psi_\mathbf{s}^\v$~is idempotent with kernel~$U_\mathbf{s}$.

The method is as follows. By construction (see Notation~\ref{N9.12}), $\Psi_\mathbf{s}^\v$~is a composite,
$$\Psi_\mathbf{s}^\v=
(\psi_{1,\mathbf{s}}^\v\psi_{2,\mathbf{s}}^\v\ldots\psi_{|\mathbf{s}|-1,\mathbf{s}}^\v)
\ldots
(\psi_{1,\mathbf{s}}^\v\psi_{2,\mathbf{s}}^\v\psi_{3,\mathbf{s}}^\v)
(\psi_{1,\mathbf{s}}^\v\psi_{2,\mathbf{s}}^\v)
(\psi_{1,\mathbf{s}}^\v).$$
We will construct subcomplexes~$U_\mathbf{s}^j$ of $\Delta(|\mathbf{s}|)$ for $0\leq j<|\mathbf{s}|$ such that
\begin{align*}
&U_\mathbf{s}=U_\mathbf{s}^0+\ldots+U_\mathbf{s}^{|\mathbf{s}|-1},\\
&\psi_{i,\mathbf{s}}^\v U_\mathbf{s}^j\subset U_\mathbf{s}^j\quad 
(0<i<j<|\mathbf{s}|),\\
&\psi_{j,\mathbf{s}}^\v U_\mathbf{s}^j\subset U_\mathbf{s}^{j-1}\quad 
(0<j<|\mathbf{s}|),\\
&U_\mathbf{s}^0=0,
\end{align*}
from which it will follow that 
$U_\mathbf{s}\subset\ker\Psi_\mathbf{s}^\v$.
We will also show that
$$(\psi_{i,\mathbf{s}}^\v-\id)\Delta(|\mathbf{s}|)\subset U_\mathbf{s}\quad 
(0<i<|\mathbf{s}|),$$
from which it will follow that 
$(\Psi_\mathbf{s}^\v-\id)\Delta(|\mathbf{s}|)\subset U_\mathbf{s}$.
From these inclusions it will indeed follow that $\Psi_\mathbf{s}^\v$~is idempotent with kernel~$U_\mathbf{s}$ as required.

The subcomplexes~$U_\mathbf{s}^j$ are defined by induction on the number of terms in~$\mathbf{s}$. In the many-term case $\mathbf{s}=(\mathbf{s'},q,p)$ recall from Notation~\ref{N8.4} that 
$$U_\mathbf{s}
=(\d_{|\mathbf{s}|-p+1}^\v)^{p-q}U_\mathbf{s'}
+(\d_0^\v)^{|\mathbf{s}|-p}U_{(p)}
+V_\mathbf{s},$$
where $V_\mathbf{s}$~is the subcomplex of $\Delta(|\mathbf{s}|)$ generated by the basis elements
$$[i_0,\ldots,i_{r-1},|\mathbf{s}|-p,i_{r+1},\ldots,i_m]$$
with $0\leq i_{r-1}<|\mathbf{s}|-p<i_{r+1}\leq |\mathbf{s}|-q$.

\begin{notation} \label{N11.1}
For $0\leq j<p$ let $U_{(p)}^j$ be the subcomplex of $\Delta(p)$ generated by the basis elements $[i_0,\ldots,i_m]$ with at least two terms less than or equal to~$j$ and with no term equal to $j+1$.

For $\mathbf{s}=(\mathbf{s'},q,p)$ with $q=0$, let
$$U_\mathbf{s}^j=
\begin{cases}
(\d_{|\mathbf{s}|-p+1}^\v)^{p-q}U_\mathbf{s'}^j& 
(0\leq j<|\mathbf{s}|-p),\\
(\d_0^\v)^{|\mathbf{s}|-p}U_{(p)}^{j-|\mathbf{s}|+p}
+ V_\mathbf{s}&
(|\mathbf{s}|-p\leq j<|\mathbf{s}|).
\end{cases}$$

For $\mathbf{s}=(\mathbf{s'},q,p)$ with $q>0$, let
$$
U_\mathbf{s}^j=\begin{cases}
(\d_{|\mathbf{s}|-p+1}^\v)^{p-q}U_\mathbf{s'}^j& 
(0\leq j<|\mathbf{s}|-p),\\
(\d_{|\mathbf{s}|-p+1}^\v)^{p-q}U_\mathbf{s'}^{|\mathbf{s}|-p}\\
\qquad{}+(\d_0^\v)^{|\mathbf{s}|-p}U_{(p)}^{j-|\mathbf{s}|+p}
+V_\mathbf{s}
&(|\mathbf{s}|-p\leq j<|\mathbf{s}|-q),\\
(\d_{|\mathbf{s}|-p+1}^\v)^{p-q}U_\mathbf{s'}^{j-p+q}\\
\qquad{}+(\d_0^\v)^{|\mathbf{s}|-p}U_{(p)}^{j-|\mathbf{s}|+p}
+V_\mathbf{s}&
(|\mathbf{s}|-q\leq j<|\mathbf{s}|).
\end{cases}$$
\end{notation}

We begin with the following result.

\begin{proposition} \label{P11.2}
Let $\mathbf{s}$ be an up-down vector. Then
\begin{align*}
&U_\mathbf{s}=U_\mathbf{s}^0+\ldots+U_\mathbf{s}^{|\mathbf{s}|-1},\\
&U_\mathbf{s}^0=0.
\end{align*}
\end{proposition}

\begin{proof}
We use induction  on the number of terms in~$\mathbf{s}$. The inductive step is obvious; we will therefore consider the one-term case $\mathbf{s}=(p)$. 

Recall from Notation~\ref{N8.4} that $U_{(p)}$~is the subcomplex of $\Delta(p)$ generated by the basis elements $[i_0,\ldots,i_m]$ with $m>0$ and $i_1\leq p-m$. It is easy to see that the generating set for~$U_{(p)}$ is the union of the generating sets for $U_{(p)}^0,\ldots,U_{(p)}^{p-1}$; therefore $U_{(p)}=U_{(p)}^0+\ldots+U_{(p)}^{p-1}$. It is also easy to see that the generating set for~$U_{(p)}^0$ is empty; therefore $U_{(p)}^0=0$.

This completes the proof.
\end{proof}

We will now give three lemmas aimed at describing the subcomplexes~$V_\mathbf{s}$ more explicitly.

\begin{lemma} \label{L11.3}
If $k$~is a fixed integer with $0\leq k\leq m$ then $\Delta(m)$ is generated as a chain complex by the basis elements $[j_0,\ldots,j_m]$ including a term equal to~$k$.
\end{lemma} 

\begin{proof}
Let $a$ be a basis element not including~$k$, and let $b$ be the basis element obtained by inserting a term equal to~$k$ in~$a$. Then $\d b$ has a term equal to~$a$, and every other term of $\d b$ includes~$k$. The result follows.
\end{proof}

\begin{lemma} \label{L11.4}
If $\mathbf{s}=(\mathbf{s'},q,p)$ then
$$[w_{|\mathbf{s}|-p,p-q}^\v-\id]\Delta(|\mathbf{s}|)\subset V_\mathbf{s}.$$
\end{lemma}

\begin{proof}
Let $a$ be a basis element for $\Delta(|\mathbf{s}|)$ which includes the term $|\mathbf{s}|-p$. By Proposition~\ref{P10.4}, if $a$~is a generator for~$V_\mathbf{s}$ then $w_{|\mathbf{s}|-p,p-q}^\v a=0$; if $a$~is not a generator for~$V_\mathbf{s}$ then $w_{|\mathbf{s}|-p,p-q}^\v a=a$. The result now follows because of Lemma~\ref{L11.3}.
\end{proof}

\begin{lemma} \label{L11.5}
Let $\mathbf{s}$ be an up-down vector with last term~$p$ and let $A$ be the set of basis elements $[j_0,\ldots,j_m]$ for $\Delta(|\mathbf{s}|)$ with at least two terms in the set
$$\{\,0,1,\ldots,|\mathbf{s}|-p\,\}.$$
For $0\leq j<|\mathbf{s}|$ let $B^j$ be the set of basis elements with at least two terms in the set
$$\{0,1,\ldots,j\}$$
and with no term $j+1$, and let $C^j$ be the set of basis elements with at least two terms in the set
$$\{\,|\mathbf{s}|-p,|\mathbf{s}|-p+1,\ldots,j\}$$
and with no term $j+1$.
Then $U_\mathbf{s}^j$~is generated as a chain complex by a subset of $A\cup B^j$, and $U_\mathbf{s}^j$~contains every member of~$C^j$.
\end{lemma}

\begin{proof}
The proof is by induction on the number of terms in~$\mathbf{s}$. 

If $\mathbf{s}=(p)$ then the results hold because $U_\mathbf{s}^j$~is generated by the members of~$C^j$ and because $B^j=C^j$.

From now on let $\mathbf{s}=(\mathbf{s'},q,p)$ and let $p'$ be the last term of~$\mathbf{s'}$. We will first show that $U_\mathbf{s}^j$ is generated by some of the members of $A\cup B^j$. We do this by considering the various constituents of~$U_\mathbf{s}^j$.

Suppose that $0\leq j<|\mathbf{s}|-p$. Since $|\mathbf{s'}|-p'<|\mathbf{s}|-p$, it follows from the inductive hypothesis that $U_\mathbf{s'}^j$~is generated by basis elements with at least two terms less than or equal to $|\mathbf{s}|-p$, and it then follows that $(\d_{|\mathbf{s}|-p+1}^\v)^{p-q}U_\mathbf{s'}^j$ is generated by members of~$A$.

Suppose that $q>0$ and $|\mathbf{s}|-p\leq j<|\mathbf{s}|-q$. Then $(\d_{|\mathbf{s}|-p+1}^\v)^{p-q}U_\mathbf{s'}^{|\mathbf{s}|-p}$ is generated by members of~$A$ as in the previous case.

Suppose that $|\mathbf{s}|-q\leq j<|\mathbf{s}|$. Then $U_\mathbf{s'}^{j-p+q}$~is generated by basis elements with at least two terms less than or equal to $|\mathbf{s}|-p$, or with at least two terms less than or equal to $j-p+q$ and with no term $j-p+q+1$. It follows that $(\d_{|\mathbf{s}|-p+1}^\v)^{p-q}U_\mathbf{s'}^{j-p+q}$ is generated by members of $A\cup B^j$.

For $|\mathbf{s}|-p\leq j<|\mathbf{s}|$ it is clear that $(\d_0^\v)^{|\mathbf{s}|-p}U_{(p)}^{j-|\mathbf{s}|+p}$ is generated by members of~$B^j$.

It is also clear that $V_\mathbf{s}$~is generated by members of~$A$.

From these results it follows in all cases that $U_\mathbf{s}^j$~is generated by members of $A\cup B^j$.

Next we show that every member~$c$ of~$C^j$ is in~$U_\mathbf{s}^j$.

There is nothing to prove in cases with $0\leq j<|\mathbf{s}|-p$, because in those cases $C^j$~is empty.

From now on, suppose that $|\mathbf{s}|-p\leq j<|\mathbf{s}|$. By Lemma~\ref{L11.4}
$$\im(w_{|\mathbf{s}|-p,p-q}^\v-\id)\subset V_\mathbf{s}\subset U_\mathbf{s}^j,$$
so it suffices to show that $w_{|\mathbf{s}|-p,p-q}^\v c\in U_\mathbf{s}^j$. We do this by considering the various terms of 
\begin{multline*}
w_{|\mathbf{s}|-p,p-q}^\v c
=(\d_{|\mathbf{s}|-p+1}^\v)^{p-q}(\e_{|\mathbf{s}|-p}^\v)^{p-q}c\\
{}+(\d_{|\mathbf{s}|-p+1}^\v)^{p-q}(\d_0^\v)^{|\mathbf{s}|-p}(\e_0^\v)^{|\mathbf{s}|-q}c
+(\d_0^\v)^{|\mathbf{s}|-p}(\e_0^\v)^{|\mathbf{s}|-p}c
\end{multline*}
(see Proposition~\ref{P10.4}).

If $|\mathbf{s}|-p\leq j<|\mathbf{s}|-q$ then it follows from the inductive hypothesis that the first term is zero or is in $(\d_{|\mathbf{s}|-p+1}^\v)^{p-q}U_\mathbf{s'}^{|\mathbf{s}|-p}$; if $|\mathbf{s}|-q\leq j<|\mathbf{s}|$ then it similarly follows from the inductive hypothesis that the first term is zero or is in $(\d_{|\mathbf{s}|-p+1}^\v)^{p-q}U_\mathbf{s'}^{j-p+q}$; in any case we see that the first term is in~$U_\mathbf{s}^j$.

For all~$j$ with $|\mathbf{s}|-p\leq j<|\mathbf{s}|$ the second and third terms are zero or are in $(\d_0^\v)^{|\mathbf{s}|-p}U_{(p)}^{j-|\mathbf{s}|+p}$, so they are also in~$U_\mathbf{s}^j$.

This completes the proof.
\end{proof}

We deduce that the morphisms~$\psi_{i,\mathbf{s}}^\v$ act in the required way.

\begin{proposition} \label{P11.6}
The morphisms~$\psi_{i,\mathbf{s}}^\v$ are such that
$$(\psi_{i,\mathbf{s}}^\v-\id)\Delta(|\mathbf{s}|)
\subset U_\mathbf{s}\quad (0<i<|\mathbf{s}|).$$
\end{proposition}

\begin{proof}
We use induction on the number of terms in~$\mathbf{s}$.

Suppose that $\mathbf{s}=(\mathbf{s'},q,p)$ and $0<i<|\mathbf{s}|-p$. By Proposition~\ref{P10.5},
$$\im(\psi_{i,\mathbf{s}}^\v-\id)
\subset(\d_{|\mathbf{s}|-p+1}^\v)^{p-q}\im(\psi_{i,\mathbf{s'}}-\id)
+\im(w_{|\mathbf{s}|-p,p-q}^\v-\id).$$
By the inductive hypothesis,
$$(\d_{|\mathbf{s}|-p+1}^\v)^{p-q}\im(\psi_{i,\mathbf{s'}}-\id)
\subset(\d_{|\mathbf{s}|-p+1}^\v)^{p-q}U_\mathbf{s'}
\subset U_\mathbf{s};$$
by Lemma~\ref{L11.4},
$$\im(w_{|\mathbf{s}|-p,p-q}^\v-\id)\subset V_\mathbf{s}\subset U_\mathbf{s}.$$
Therefore $\im(\psi_{i,\mathbf{s}}^\v-\id)\subset U_\mathbf{s}$.

Now suppose that $\mathbf{s}=(\mathbf{s'},q,p)$ and $i=|\mathbf{s}|-p$. We can apply a similar argument, using Propositions~\ref{P10.6} and~\ref{P10.7}.

Finally suppose that $|\mathbf{s}|-p<i<|\mathbf{s}|$. By Notation~\ref{N9.8} and Proposition~\ref{P10.2},
$$\psi_{i,\mathbf{s}}^\v-\id=\phi_{i,1}^\v-\id=\d_{i+1}^\v(\e_{i-1}^\v-\e_i^\v).$$
Because of Lemma~\ref{L11.3}, it suffices to show that $\d_{i+1}^\v(\e_{i-1}^\v-\e_i^\v)a$ is in~$U_\mathbf{s}$ when $a$~is a basis element including $i+1$. If $a$~is a basis element of that form not including~$i$, then $\d_{i+1}^\v(\e_{i-1}^\v-\e_i^\v)a=0$; if $a$~is a basis element of that form including~$i$, then $\d_{i+1}^\v(\e_{i-1}^\v-\e_i^\v)a$ is a basis element including $i-1$ and~$i$ but not $i+1$. In view of Lemma~\ref{L11.5}, this suffices to show that $\d_{i+1}^\v(\e_{i-1}^\v-\e_i^\v)a$ is in~$U_\mathbf{s}$ in all cases, and this completes the proof.
\end{proof}

\begin{proposition} \label{P11.7}
The morphisms~$\psi_{i,\mathbf{s}}^\v$ are such that
$$\psi_{i,\mathbf{s}}^\v U_\mathbf{s}^j
\subset U_\mathbf{s}^j\quad (0<i<j<|\mathbf{s}|).$$
\end{proposition}

\begin{proof}
We use induction on the number of terms in~$\mathbf{s}$.

Suppose first that $\mathbf{s}=(p)$. By Notation~\ref{N9.8} and Proposition~\ref{P10.2}
$$\psi_{i,\mathbf{s}}^\v=\phi_{i,1}^\v=\d_{i+1}^\v(\e_{i-1}^\v-\e_i^\v)+\id.$$
According to Notation~\ref{N11.1}, the chain complex~$U_\mathbf{s}^j$~is generated by the basis elements with at least two terms less than or equal to~$j$ and with no term equal to $j+1$. The result now follows from a simple computation.

Now suppose that $\mathbf{s}=(\mathbf{s'},q,p)$, and let $p'$ be the last term of~$\mathbf{s'}$. Recall from Notation~\ref{N11.1} that $U_\mathbf{s}^j$~is a sum of constituents which may have one of the following forms:
$$(\d_{|\mathbf{s}|-p+1}^\v)^{p-q}U_\mathbf{s'}^{j'},\quad
(\d_0^\v)^{|\mathbf{s}-p}U_{(p)}^{j''},\quad
V_\mathbf{s}.$$
In almost all cases it follows straightforwardly from Propositions \ref{P10.5}--\ref{P10.9} and the inductive hypothesis that $\psi_{i,\mathbf{s}}^\v$~maps the constituents of~$U_\mathbf{s}^j$ into~$U_\mathbf{s}^j$. The exceptional cases are
$$\psi_{|\mathbf{s}|-p,\mathbf{s}}^\v
(\d_0^\v)^{|\mathbf{s}|-p}U_{(p)}^{j-|\mathbf{s}|+p}\quad 
(q>0,\ |\mathbf{s}|-p<j<|\mathbf{s}|)$$
and
$$\psi_{|\mathbf{s}|-q+1,\mathbf{s}}^\v(\d_{|\mathbf{s}|-p+1}^\v)^{p-q}U_\mathbf{s'}^{j-p+q}\quad
(|\mathbf{s}|-q+1<j<|\mathbf{s}|).$$
We deal with these cases as follows.

In the first case let $c$ be a generator for $U_{(p)}^{j-|\mathbf{s}|+p}$, so that $c$~is a basis element with at least two terms less than or equal to $j-|\mathbf{s}|+p$ and with no term equal to $j-|\mathbf{s}|+p+1$. By Proposition~\ref{P10.7},
$$\bigl[\psi_{|\mathbf{s}|-p,\mathbf{s}}^\v-\id\bigr](\d_0^\v)^{|\mathbf{s}|-p}c
=(\d_{|\mathbf{s}|-p+1}^\v)^{p-q}c',$$
where
$$c'=\d_{|\mathbf{s}|-p+1}^\v
\bigl[\e_{|\mathbf{s}|-p-1}^\v-\e_{|\mathbf{s}|-p}^\v\bigr]
(\e_{|\mathbf{s}|-p+1}^\v)^{p-q}(\d_0^\v)^{|\mathbf{s}|-p}c.$$
If $|\mathbf{s}|-p<j\leq |\mathbf{s}|-q$ then $c'=0$; if $|\mathbf{s}|-q<j<|\mathbf{s}|$ then $c'$~is a linear combination of basis elements with at least two terms in the set
$$\{\,|\mathbf{s}|-p-1,\,|\mathbf{s}|-p,\,\ldots,\,j-p+q\,\}$$
and with no term equal to $j-p+q+1$, so that $c'\in U_\mathbf{s'}^{j-p+q}$ by Lemma~\ref{L11.5}. In all cases it follows that $\psi_{|\mathbf{s}|-p}^\v\in U_\mathbf{s}^j$.

It remains to show that
$$\psi_{|\mathbf{s}|-q+1,\mathbf{s}}^\v
(\d_{\mathbf{s}|-p+1}^\v)^{p-q}U_\mathbf{s'}^{j-p+q}\subset U_\mathbf{s}^j\quad
(|\mathbf{s}|-q+1<j<|\mathbf{s}|).$$
Because of Lemma~\ref{L11.5} it suffices to show that $\psi_{|\mathbf{s}|-q+1,\mathbf{s}}^\v
(\d_{\mathbf{s}|-p+1}^\v)^{p-q}c$ is in~$U_\mathbf{s}^j$ when $c$~is a basis element in~$U_\mathbf{s'}^{j-p+q}$ with two terms less than or equal to $|\mathbf{s}|-p$, or with two terms less than or equal to $j-p+q$ and with no term $j-p+q+1$. By Lemma~\ref{L11.4}
$$\im(w_{|\mathbf{s}|-p,p-q}^\v-\id)\subset V_\mathbf{s}\subset U_\mathbf{s}^j;$$
it therefore suffices to show that $w_{|\mathbf{s}|-p,p-q}^\v\psi_{|\mathbf{s}|-q+1,\mathbf{s}}^\v
(\d_{\mathbf{s}|-p+1}^\v)^{p-q}c$ is in~$U_\mathbf{s}^j$ for each such basis element~$c$. By Proposition~\ref{P10.8}
\begin{multline*}
w_{|\mathbf{s}|-p,p-q}^\v
\psi_{|\mathbf{s}|-q+1,\mathbf{s}}^\v
(\d_{|\mathbf{s}|-p+1}^\v)^{p-q}c
=(\d_{|\mathbf{s}|-p+1}^\v)^{p-q}
\psi_{|\mathbf{s}|-p+1,\mathbf{s'}}^\v c\\
{}+(\d_0^\v)^{|\mathbf{s}|-p}
\bigl[\phi_{p-q+1,1}^\v(\d_1^\v)^{p-q}-(\d_1^\v)^{p-q}\phi_{1,1}^\v\bigr]
(\e_0^\v)^{|\mathbf{s}|-p}c.
\end{multline*}
The first of the terms on the right hand side is in~$U_\mathbf{s}^j$ by the inductive hypothesis. If $c$~has two terms less than or equal to $|\mathbf{s}|-p$, then the second term on the right hand side is zero. If $c$~has two terms less than or equal to $j-p+q$ and has no term $j-p+q+1$, then the second term on the right hand side is a linear combination of basis elements with at least two terms in the set
$$\{\,|\mathbf{s}|-p,|\mathbf{s}|-p+1,\ldots,j\,\}$$
and with no term $j+1$ and is in~$U_\mathbf{s}^j$ by Lemma~\ref{L11.5}. This completes the proof.
\end{proof}

\begin{proposition} \label{P11.8}
The morphisms~$\psi_{j,\mathbf{s}}^\v$ are such that
$$\psi_{j,\mathbf{s}}^\v U_\mathbf{s}^j
\subset U_\mathbf{s}^{j-1}\quad (0<j<|\mathbf{s}|).$$
\end{proposition}

\begin{proof}
This is similar. Again we use induction on the number of terms in~$\mathbf{s}$.

Suppose that $\mathbf{s}=(p)$. By definition, $U_\mathbf{s}^j$~is generated by the basis elements with at least two terms less than or equal to~$j$ and with no term $j+1$. If $a$~is such a basis element then 
$$\psi_{j,\mathbf{s}}^\v a
=\phi_{j,1}^\v a
=\d_{j+1}^\v(\e_{j-1}^\v-\e_j^\v)a+a,$$
and this is a linear combination of basis elements with at least two terms less than or equal to $j-1$ and with no term~$j$. Therefore $\psi_{j,\mathbf{s}}^\v U_\mathbf{s}^j\subset U_\mathbf{s}^{j-1}$.

Now suppose that $\mathbf{s}=(\mathbf{s'},q,p)$. In almost all cases it follows from  Propositions \ref{P10.5}--\ref{P10.9} and the inductive hypothesis that $\psi_{j,\mathbf{s}}^\v$~maps each constituent of~$U_\mathbf{s}^j$ into~$U_\mathbf{s}^{j-1}$. The only difficulty is to show that
$$\psi_{|\mathbf{s}|-q+1,\mathbf{s}}^\v
(\d_{\mathbf{s}|-p+1}^\v)^{p-q}U_\mathbf{s'}^{|\mathbf{s}|-p+1}\subset U_\mathbf{s}^{|\mathbf{s}|-q}$$
in the case $q>1$. By Lemma~\ref{L11.5}, it suffices to show that 
$$\psi_{|\mathbf{s}|-q+1,\mathbf{s}}^\v
(\d_{\mathbf{s}|-p+1}^\v)^{p-q}c\in U_\mathbf{s}^{|\mathbf{s}|-q}$$
when $c$~is a basis element in~$U_\mathbf{s'}^{|\mathbf{s}|-p+1}$ with two terms less than or equal to $|\mathbf{s}|-p$, or with two terms less than or equal to $|\mathbf{s}|-p+1$ and with no term $|\mathbf{s}|-p+2$. As in the proof of Proposition~\ref{P11.7} it suffices to show that
\begin{multline*}
(\d_{|\mathbf{s}|-p+1}^\v)^{p-q}
\psi_{|\mathbf{s}|-p+1,\mathbf{s'}}^\v c\\
{}+(\d_0^\v)^{|\mathbf{s}|-p}
\bigl[\phi_{p-q+1,1}^\v(\d_1^\v)^{p-q}-(\d_1^\v)^{p-q}\phi_{1,1}^\v\bigr]
(\e_0^\v)^{|\mathbf{s}|-p}c
\end{multline*}
is in~$U_\mathbf{s}^{|\mathbf{s}|-q}$.

The first of these terms is in~$U_\mathbf{s}^{|\mathbf{s}|-q}$ by the inductive hypothesis. If $c$~has two terms less than or equal to $|\mathbf{s}|-p$, then the second term is zero. If $c$~has two terms less than or equal to $|\mathbf{s}|-p+1$ and has no term $|\mathbf{s}|-p+2$, then the second term is a linear combination of basis elements with terms $|\mathbf{s}|-p$ and $|\mathbf{s}|-q$ and with no term $|\mathbf{s}|-q+1$, and is therefore in~$U_\mathbf{s}^{|\mathbf{s}|-q}$ by Lemma~\ref{L11.5}. This completes the proof.
\end{proof}

It follows from Propositions \ref{P11.2} and \ref{P11.6}--\ref{P11.8} that $\Psi_\mathbf{s}^\v$~is idempotent with kernel~$U_\mathbf{s}$. Recall from Theorem~\ref{T6.11} that if $X$~is a set with complicial identities then $X_{|\mathbf{s}|}\cong\Hom[\lambda\Delta(|\mathbf{s}|),X]$. We draw the following conclusions.

\begin{proposition} \label{P11.9}
Let $\mathbf{s}$ be an up-down vector and let $X$ be a set with complicial identities. Then $\Psi_\mathbf{s}$~is an idempotent operation on~$X_{|\mathbf{s}|}$. There is a natural bijection
$$\Psi_\mathbf{s}X_{|\mathbf{s}|}\cong\Hom[\lambda S_\mathbf{s},X],$$
where
$$S_\mathbf{s}=\Delta(|\mathbf{s}|)/U_\mathbf{s},$$
and the inclusion of $\Psi_\mathbf{s}X_{|\mathbf{s}|}$ in
$$X_{|\mathbf{s}|}\cong\Hom[\lambda\Delta(|\mathbf{s}|),X]$$
is induced by the quotient homomorphism $\Delta(|\mathbf{s}|)\to S_\mathbf{s}$. The image $\Psi_\mathbf{s}X_{|\mathbf{s}|}$ is the subset of~$X_{|\mathbf{s}|}$ consisting of the elements~$x$ such that
$$\psi_{1,\mathbf{s}}x=\ldots=\psi_{|\mathbf{s}|-1,\mathbf{s}}x=x.$$
\end{proposition}

\begin{proof}
We need only prove the final statement. To do this we first observe that if $0<i<|\mathbf{s}|$ then
$$\im(\psi_{i,\mathbf{s}}^\v-\id)\subset U_\mathbf{s}=\ker\Psi_\mathbf{s}^\v$$
by Proposition~\ref{P11.6}, hence $\psi_{i,\mathbf{s}}\Psi_\mathbf{s}=\Psi_\mathbf{s}$. We then recall from Notation~\ref{N9.12} that $\Psi_\mathbf{s}$~is an iterated composite of the operations~$\psi_{i,\mathbf{s}}$. The result follows.
\end{proof}

\section{The pull-back property} \label{S12}

Let $X$ be a set with complicial identities. According to Proposition~\ref{P11.9} there is a functor $S_\mathbf{s}\mapsto\Psi_\mathbf{s}X_{|\mathbf{s}|}$ from simple chain complexes to sets. According to Theorem~\ref{T8.12} (see also Notations \ref{N8.3} and~\ref{N8.9}) the image of an $\mathbf{s}$-simple square is given by
$$\xymatrix{
\Psi_\mathbf{s}X_{|\mathbf{s}|}
\ar_{\d_{|\mathbf{s}|-p+1}^{p-q}}[d] 
\ar^{\d_0^{|\mathbf{s}|-p}}[r]
& \Psi_{(p)}X_p
\ar^{\d_1^{p-q}}[d]
\\
\Psi_\mathbf{s'}X_{|\mathbf{s'}|}
\ar_{\d_0^{|\mathbf{s}|-p}}[r]
& \Psi_{(q)}X_q.
}$$
In this section we show that the functor yields an $\omega$-category; that is, we show that the images of simple squares are pull-back squares (see Proposition~\ref{P5.8}). 

Recall from Proposition~\ref{P11.9} that the image of~$\Psi_\mathbf{s}$ is the intersection of the fixed point sets of the operations~$\psi_{i,\mathbf{s}}$. We begin the proof by making the following observations.

\begin{proposition} \label{P12.1}
Let $\mathbf{s}$ be an up-down vector with last term~$p$ and let $k$ be an integer with $|\mathbf{s}|-p\leq k<|\mathbf{s}|$. Then 
$$\psi_{k+1,\mathbf{s}}x=\psi_{k+2,\mathbf{s}}x=\ldots
=\psi_{|\mathbf{s}|-1,\mathbf{s}}x=x$$
if and only if
$$\d_{k+2}x=\e_k\d_{k+1}^2 x,\ \d_{k+3}x=\e_k^2\d_{k+1}^3 x,\ \ldots,\ 
\d_{|\mathbf{s}|}x=\e_k^{|\mathbf{s}|-k-1}\d_{k+1}^{|\mathbf{s}|-k}x.$$
\end{proposition}

\begin{proof}
For $k<i<|\mathbf{s}|$ we have $\psi_{i,\mathbf{s}}=\phi_{i,1}$ (Notation~\ref{N9.8}), hence, by Proposition~\ref{P9.3},
$$\psi_{i,\mathbf{s}}x=x\iff \d_{i+1}x\in\im\e_{i-1}.$$
Note also that $\d_i\e_{i-1}=\id$ and $\d_i\d_{i+1}=\d_i\d_i$, hence
$$\d_{i+1}x\in\im\e_{i-1}
\iff \d_{i+1}x=\e_{i-1}\d_i\d_{i+1}x
\iff\d_{i+1}x=\e_{i-1}\d_i\d_i x.$$

It follows from this that if $\psi_{i,\mathbf{s}}x=x$ for all~$i$ with $k<i<|\mathbf{s}|$ then
\begin{align*}
&\d_{k+2}x=\e_k\d_{k+1}\d_{k+1}x=\e_k\d_{k+1}^2 x,\\
&\d_{k+3}x
=\e_{k+1}\d_{k+2}\d_{k+2}x
=\e_{k+1}\e_k\d_{k+2}\d_{k+1}^2 x
=\e_k^2\d_{k+1}^3 x,\\
&\ldots,\\
&\d_{|\mathbf{s}|}x=\e_k^{|\mathbf{s}|-k-1}\d_{k+1}^{|\mathbf{s}|-k}x.
\end{align*}
Conversely, if $\d_{i+1}x=\e_k^{i-k}\d_{k+1}^{i-k+1}x$ for all~$i$ with $k<i<|\mathbf{s}|$ then $\d_{i+1}x\in\im\e_{i-1}$ for all~$i$ with $k<i<|\mathbf{s}|$, hence $\psi_{i,\mathbf{s}}x=x$ for all~$i$ with $k<i<|\mathbf{s}|$. This completes the proof.
\end{proof}

\begin{proposition} \label{P12.2}
Let $x$ be an $n$-dimensional element in the image of an operation~$w_{k,l}$. 
If $k<i\leq k+l$ then
$$[\d_{i+1}x=\e_k^{i-k}\d_{k+1}^{i-k+1}x]
\iff[\d_{i-k+1}\d_0^k x=\e_0^{i-k}\d_1^{i-k+1}\d_0^k x].$$
If $k+l<i<n$ then
\begin{multline*}
[\d_{i+1}x=\e_k^{i-k}\d_{k+1}^{i-k+1}x] 
\iff [\d_{i-l+1}\d_{k+1}^l x=\e_k^{i-k-l}\d_{k+1}^{i-k-l+1}\d_{k+1}^l x\\
\textit{and}\ 
\d_{i-k+1}\d_0^k x=\e_0^{i-k}\d_1^{i-k+1}\d_0^k x].
\end{multline*}
\end{proposition}

\begin{proof}
Recall from Proposition~\ref{P10.4} that
$$w_{k,l}^\v
=(\d_{k+1}^\v)^l(\e_k^\v)^l
-(\d_{k+1}^\v)^l(\d_0^\v)^k(\e_0^\v)^{k+l}
+(\d_0^\v)^k(\e_0^\v)^k.$$
Direct computations show that
$$\d_{i+1}w_{k,l}=\begin{cases}
w_{k,l-1}\d_{i+1}& (k<i<k+l),\\
w_{k,l}\d_{i+1}& (k+l\leq i<n)
\end{cases}$$
and
$$\e_k^{i-k}\d_{k+1}^{i-k+1}w_{k,l}=\begin{cases}
w_{k,l-1}\e_k^{i-k}\d_{k+1}^{i-k+1}& (k<i<k+l),\\
w_{k,l}\e_k^{i-k}\d_{k+1}^{i-k+1}& (k+l\leq i<n).
\end{cases}$$
Recall from Notation~\ref{N9.6} and Proposition~\ref{P9.7} that
$$w_{k,l}x=\d_{k+1}^l x\w_{k,l}\d_0^k x,\ 
\d_{k+1}^l w_{k,l}x=\d_{k+1}^l x,\
\d_0^k w_{k,l}x=\d_0^k x.$$
For $x\in\im w_{k,l}$ and for $k<i<k+l$ it follows that $\d_{i+1}x=\e_k^{i-k}\d_{k+1}^{i-k+1}x$ if and only if
$$\d_{k+1}^{l-1}\d_{i+1}x=\d_{k+1}^{l-1}\e_k^{i-k}\d_{k+1}^{i-k+1}x,\quad
\d_0^k\d_{i+1}x=\d_0^k\e_k^{i-k}\d_{k+1}^{i-k+1}x;$$
for $k+l\leq i<n$ it follows that $\d_{i+1}x=\e_k^{i-k}\d_{k+1}^{i-k+1}x$ if and only if
$$\d_{k+1}^l\d_{i+1}x=\d_{k+1}^l\e_k^{i-k}\d_{k+1}^{i-k+1}x,\quad
\d_0^k\d_{i+1}x=\d_0^k\e_k^{i-k}\d_{k+1}^{i-k+1}x.$$
This gives the result. (In cases with $k<i\leq k+l$ the first condition is omitted from the statement of the proposition because it is satisfied automatically).
\end{proof}

\begin{proposition} \label{P12.3}
Let $\mathbf{s}=(\mathbf{s'},q,p)$ be an up-down vector with more than one term, let $X$ be a set with complicial identities, let $x$ be be a member of~$X_{|\mathbf{s}|}$ such that $w_{|\mathbf{s}|-p,p-q}x=x$, and let
$y=\d_{|\mathbf{s}|-p+1}^{p-q}x$, $z=\d_0^{|\mathbf{s}|-p}x$. Then $x\in\Psi_\mathbf{s}X_{|\mathbf{s}|}$ if and only if $y\in\Psi_\mathbf{s'}X_{|\mathbf{s'}|}$ and $z\in\Psi_{(p)}X_p$.
\end{proposition}

\begin{proof}
According to Proposition~\ref{P11.9} we must show that
$$\psi_{1,\mathbf{s}}x=\ldots=\psi_{|\mathbf{s}|-1,\mathbf{s}}x=x$$
if and only if
$$\psi_{1,\mathbf{s'}}y=\ldots=\psi_{|\mathbf{s'}|-1,\mathbf{s'}}y=y,\quad
\psi_{1,(p)}z=\ldots=\psi_{p-1,(p)}z=\d_0^{|\mathbf{s}|-p}z.$$
We will consider $\psi_{i,\mathbf{s}}x$ for $i>|\mathbf{s}|-p$, then for $i<|\mathbf{s}|-p$, then for $i=|\mathbf{s}|-p$.

Let $p'$ be the last term in~$\mathbf{s'}$, and recall that $|\mathbf{s}|-p>|\mathbf{s'}|-p'$. It follows from Propositions \ref{P12.1} and~\ref{P12.2}  that
$$\psi_{|\mathbf{s}|-p+1,\mathbf{s}}x=\ldots=\psi_{|\mathbf{s}|-1,\mathbf{s}}x=x$$
if and only if
$$\psi_{|\mathbf{s}|-p+1,\mathbf{s'}}y=\ldots=\psi_{|\mathbf{s'}|-1,\mathbf{s'}}y=y,\quad
\psi_{1,(p)}z=\ldots=\psi_{p-1,(p)}z=\d_0^{|\mathbf{s}|-p}z.$$

Recall that $x=y\w_{|\mathbf{s}|-p,p-q}z$ (Notation~\ref{N9.6}). For $0<i<|\mathbf{s}|-p$ we have
$\psi_{i,\mathbf{s}}x=\psi_{i,\mathbf{s'}}y\w_{|\mathbf{s}|-p,p-q}z$
(Notation~\ref{N9.11}), hence $\psi_{i,\mathbf{s}}x=x$ if and only if $\psi_{i,\mathbf{s'}}y=y$. 

In the case $q=0$ we have $\psi_{|\mathbf{s}|-p,\mathbf{s}}x=w_{|\mathbf{s}|-p,p-q}x=x$
by hypothesis (see Notation~\ref{N9.9}). This completes the proof in the case $q=0$.

From now on suppose that $q>0$. It suffices to show that 
$$\psi_{|\mathbf{s}|-p,\mathbf{s}}x=x\iff\psi_{|\mathbf{s}|-p,\mathbf{s'}}y=y.$$
Equivalently, since $|\mathbf{s}|-p>|\mathbf{s'}|-p'$, it suffices to show that
$$w_{|\mathbf{s}|-p,p-q}\phi_{|\mathbf{s}|-p,p-q+1}x=x\iff\phi_{|\mathbf{s}|-p,1}y=y$$
(see Notations \ref{N9.9} and~\ref{N9.8}).

To do this, suppose first that $w_{|\mathbf{s}|-p,p-q}\phi_{|\mathbf{s}|-p,p-q+1}x=x$. Then
\begin{align*}
\d_{|\mathbf{s}|-p+1}y
&=\d_{|\mathbf{s}|-p+1}^{p-q+1}x\\
&=\d_{|\mathbf{s}|-p+1}^{p-q}w_{|\mathbf{s}|-p,p-q}\phi_{|\mathbf{s}|-p,p-q+1}x\\
&=\d_{|\mathbf{s}|-p+1}^{p-q}\phi_{|\mathbf{s}|-p,p-q+1}x
\end{align*}
by Proposition~\ref{P9.7}, hence $\d_{|\mathbf{s}|-p+1}y\in\im\e_{|\mathbf{s}|-p-1}$ (Proposition~\ref{P9.2}), hence $\phi_{|\mathbf{s}|-p,1}y=y$ (Proposition~\ref{P9.3}).

Conversely, suppose that $\phi_{|\mathbf{s}|-p,1}y=y$. By Proposition~\ref{P9.3} 
$$\d_{|\mathbf{s}|-p+1}^{p-q+1}x=\d_{|\mathbf{s}|-p+1}y\in\im\e_{|\mathbf{s}|-p-1},$$
hence
$$w_{|\mathbf{s}|-p,p-q}\phi_{|\mathbf{s}|-p,p-q+1}x
=w_{|\mathbf{s}|-p,p-q}x
=x.$$

This completes the proof.
\end{proof}

\begin{proposition} \label{P12.4}
If $\mathbf{s}=(\mathbf{s'},q,p)$ and if $X$~is a set with complicial identities, then the square
$$\xymatrix{
\Psi_\mathbf{s}X_{|\mathbf{s}|}
\ar_{\d_{|\mathbf{s}|-p+1}^{p-q}}[d] 
\ar^{\d_0^{|\mathbf{s}|-p}}[rr]
&& \Psi_{(p)}X_p
\ar^{\d_1^{p-q}}[d]
\\
\Psi_\mathbf{s'}X_{|\mathbf{s'}|}
\ar_{\d_0^{|\mathbf{s}|-p}}[rr]
&& \Psi_{(q)}X_q.
}$$
is a pull-back square of sets.
\end{proposition}

\begin{proof}
By Proposition~\ref{P9.7} there is a pull-back square
$$\xymatrix{
w_{|\mathbf{s}|-p,p-q}X_{|\mathbf{s}|}
\ar_{\d_{|\mathbf{s}|-p+1}^{p-q}}[d] 
\ar^{\d_0^{|\mathbf{s}|-p}}[rr]
&& X_p
\ar^{\d_1^{p-q}}[d]
\\
X_{|\mathbf{s'}|}
\ar_{\d_0^{|\mathbf{s}|-p}}[rr]
&& X_q
}$$
with $w_{|\mathbf{s}|-p,p-q}$ idempotent. By Proposition~\ref{P11.9} and Notation~\ref{N9.9}, if $x\in\Psi_\mathbf{s}X_{|\mathbf{s}|}$ then
$$w_{|\mathbf{s}|-p,p-q}x
=w_{|\mathbf{s}|-p,p-q}\psi_{|\mathbf{s}|-p,\mathbf{s}}x
=\psi_{|\mathbf{s}|-p,\mathbf{s}}x
=x.$$
The result now follows from Proposition~\ref{P12.3}.
\end{proof}

For a set with complicial identities~$X$ it now follows from Proposition~\ref{P11.9} that the functor
$$S\mapsto\Hom[\lambda S,X]$$
takes simple squares of chain complexes to pull-back squares of sets. By Proposition~\ref{P5.8} this determines an $\omega$-category functorially in~$X$. We will use the following notation.

\begin{notation} \label{N12.5}
Let $\beta$ be the functor from sets with complicial identities to $\omega$-categories such that
$$\Hom[\nu S,\beta X]=\Hom[\lambda S,X]$$
for every simple chain complex~$S$.
\end{notation}

\section{The equivalence} \label{S13}

We have constructed functors $\alpha$~and~$\beta$ between $\omega$-categories and sets with complicial identities (see Notations \ref{N7.3} and~\ref{N12.5}). We will now show that these functors are inverse equivalences.

In particular we must show that $\alpha\beta X\cong X$ for every set with complicial identities~$X$. We will do this by showing that
$$\Psi_{(m)}(\alpha\beta X)_m\cong\Psi_{(m)}X_m\quad (m\geq 0).$$
We must therefore show that a set with complicial identities is determined up to isomorphism by the images of the operations~$\Psi_{(m)}$. We will do this by induction on~$m$: for $m>0$ we will show that an $m$-dimensional element~$c$ is equivalent to the family consisting of the image $\Psi_{(m)}c$ and of its faces.

Recall that $\Psi_{(m)}$~is a composite of the operations~$\psi_{i,(m)}$, where
$$\psi_{i,(m)}x=\phi_{i,1}x=\d_{i+1}(\e_{i-1}\d_{i+1}x\w_i x)$$
(see Notations \ref{N9.12}, \ref{N9.8} and~\ref{N9.1}). We begin by considering an individual operation~$\phi_{i,1}$.

\begin{proposition} \label{P13.1}
Let $X$ be a set with complicial identities, let $m$~and~$i$ be integers with $0<i<m$, and let $T$ be the set of triples $(a,y,z)$ in $\phi_{i,1}X_m\times X_{m-1}^2$ such that
$$\d_{i-1}y=\d_i z,\quad
\d_{i-1}a=\d_i(y\w_{i-1}z),\quad \d_{i+1}a=\e_{i-1}\d_i y.$$
Then there is a bijection $f\colon X_m\to T$ given by
$$f(c)=(\phi_{i,1}c,\,\d_{i+1}c,\,\d_{i-1}c).$$
\end{proposition}

\begin{proof}
It follows from the axioms (Definition~\ref{D3.1}) that the formula for~$f$ defines a function whose image is contained in~$T$. We will show that there is an inverse function $g\colon T\to X_m$ given by
$$g(a,y,z)=\d_i[a\w_{i-1}(y\w_{i-1}z)];$$
the wedges in this formula exist because 
$$\d_{i-1}y=\d_i z,\quad \d_{i-1}a=\d_i(y\w_{i-1}z).$$

First we show that $gf(c)=c$ for $c\in X_m$. Let
$$A=\e_{i-1}\d_{i+1}c\w_i c=(\e_{i-1}\d_i\d_{i+1}c\w_{i-1}\d_{i+1}c)\w_i c,$$
so that $\phi_{i,1}c=\d_{i+1}A$. By Definition \ref{D3.1}(6),
$$A=\d_{i+1}A\w_{i-1}(\d_{i+1}c\w_{i-1}\d_{i-1}c);$$
therefore
\begin{align*}
gf(c)
&=\d_i[\phi_{i,1}c\w_{i-1}(\d_{i+1}c\w_{i-1}\d_{i-1}c]\\
&=\d_i[\d_{i+1}A\w_{i-1}(\d_{i+1}c\w_{i-1}\d_{i-1}c)]\\
&=\d_i A\\
&=c.
\end{align*}

Conversely we will show that $fg(a,y,z)=(a,y,z)$ for $(a,y,z)$ in~$T$. Let
$$B=a\w_{i-1}(y\w_{i-1}z),$$ 
so that $g(a,y,z)=\d_i B$. Then
\begin{align*}
\d_{i+1}g(a,y,z)
&=\d_{i+1}\d_i B\\
&=\d_i\d_{i+2}B\\
&=\d_i[\d_{i+1}a\w_{i-1}\d_{i+1}(y\w_{i-1}z)]\\
&=\d_i(\e_{i-1}\d_i y\w_{i-1}y)\\
&=\d_i\e_{i-1}y\\
&=y,
\end{align*}
and
$$\d_{i-1}g(a,y,z)
=\d_{i-1}\d_i B
=\d_{i-1}\d_{i-1}B
=\d_{i-1}(y\w_{i-1}z)
=z.$$
We also deduce that
$$\e_{i-1}\d_{i+1}g(a,y,z)
=\e_{i-1}y
=\e_{i-1}\d_i y\w_{i-1}y\\
=\d_{i+1}a\w_{i-1}y.$$
By Definition \ref{D3.1}(5),
$$B=(\d_{i+1}a\w_{i-1}y)\w_i\d_i B,$$
hence
\begin{align*}
\phi_{i,1}g(a,y,z)
&=\d_{i+1}[\e_{i-1}\d_{i+1}g(a,y,z)\w_i g(a,y,z)]\\
&=\d_{i+1}[(\d_{i+1}a\w_{i-1}y)\w_i\d_i B]\\
&=\d_{i+1}B\\
&=a.
\end{align*}
Therefore $fg(a,y,z)=(a,y,z)$.

This completes the proof.
\end{proof}

According to this proposition, if $X$~is a set with complicial identities and if $0<i<m$, then an $m$-dimensional member~$c$ of~$X$ can be recovered from the image $\phi_{i,1}c$ and the faces $\d_{i+1}c$, $\d_{i-1}c$. The triples $(\phi_{i,1}c,\d_{i+1}c,\d_{i-1}c)$ that can occur are those permitted by the formulae
$$\d_{i+1}\phi_{i,1}c=\e_{i-1}\d_i\d_{i+1}c,\quad 
\d_{i-1}\phi_{i,1}c=\d_i(\d_{i+1}c\w_i\d_{i-1}c).$$
We extend this as follows.

\begin{proposition} \label{P13.2}
Let $X$ be a set with complicial identities. Then $X_0=\Psi_{(0)}X_0$. For $m>0$ the function on~$X_m$ given by
$$c\mapsto(\Psi_{(m)}c,\d_0 c,\ldots,\d_m c)$$
is injective. The image consists of the $(m+2)$-tuples
$$(a,u_0,\ldots,u_m)\in \Psi_{(m)}X_{(m)}\times X_{m-1}\times\ldots\times X_{m-1}$$
such that
$$\d_i a=F_i(u_0,\ldots,u_m)\quad (0<i<m),$$
where $F_i$~is the operation such that
$$\d_i\Psi_{(m)}c=F_i(\d_0 c,\ldots,\d_m c)$$
for all~$c$.
\end{proposition}

\begin{proof}
Recall from Notations \ref{N9.12} and~\ref{N9.8} that $\Psi_{(m)}$~is a composite of the operations~$\phi_{i,1}$ (an empty composite in the case $m=0$). The result therefore follows from Proposition~\ref{P13.2}.
\end{proof}

\begin{theorem} \label{T13.3}
The categories of $\omega$-categories and sets with complicial identities are equivalent under the functors $\alpha$~and~$\beta$.
\end{theorem}

\begin{proof}
Let $C$ be an $\omega$-category and let $X$ be a set with complicial identities. We will construct natural isomorphisms
$$\beta\alpha C\cong C,\quad \alpha\beta X\cong X.$$

Let $m$ be a nonnegative integer. By Proposition~\ref{P5.4} and Definition~\ref{D2.5}, $S_{(m)}$ is a free $\omega$-category on one $m$-dimensional generator. By Theorem~\ref{T6.11} $\lambda\Delta(m)$ is a free set with complicial identities on one $m$-dimensional generator. It follows that
$$C_m\cong\Hom[\nu S_{(m)},C],\quad X_m\cong\Hom[\lambda\Delta(m),X];$$
recall also from Propositiion~\ref{P11.9} that $\Hom[\lambda S_{(m)},X]\cong\Psi_{(m)}X_m$. It is convenient to write
$$\Psi_{(m)}\Hom[\lambda\Delta(m),X]
=\{\,x\in\Hom[\lambda\Delta(m),X]:x(\lambda\Psi_{(m)}^\v)=x\,\},$$
so that
$$\Hom[\lambda S_{(m)},X]\cong\Psi_{(m)}\Hom[\lambda\Delta(m),X].$$
Analogously we will write
$$\Psi_{(m)}\Hom[\nu\Delta(m),C]
=\{\,x\in\Hom[\nu\Delta(m),C]:x(\nu\Psi_{(m)}^\v)=x\,\},$$
so that
$$\Hom[\nu S_{(m)},C]\cong\Psi_{(m)}\Hom[\nu\Delta(m),C].$$
It follows from Notations \ref{N7.3} and~\ref{N12.5} that
\begin{align*}
(\beta\alpha C)_m
&\cong\Hom[\nu S_{(m)},\beta\alpha C]\\
&\cong\Hom[\lambda S_{(m)},\alpha C]\\
&\cong\Psi_{(m)}\Hom[\lambda\Delta(m),\alpha C]\\
&\cong\Psi_{(m)}\Hom[\nu\Delta(m),C]\\
&\cong\Hom[\nu S_{(m)},C]\\
&\cong C_m,
\end{align*}
hence $\beta\alpha C\cong C$. Analogously 
\begin{align*}
\Psi_{(m)}(\alpha\beta X)_m
&\cong\Hom[\lambda S_{(m)},\alpha\beta X]\\
&\cong\Psi_{(m)}\Hom[\lambda\Delta(m),\alpha\beta X]\\
&\cong\Psi_{(m)}\Hom[\nu\Delta(m),\beta X]\\
&\cong\Hom[\nu S_{(m)},\beta X]\\
&\cong\Hom[\lambda S_{(m)},X]\\
&\cong\Psi_{(m)}X_m,
\end{align*}
hence $\alpha\beta X\cong X$ by Proposition~\ref{P13.2}. This completes the proof.
\end{proof}

\end{document}